\documentclass[reqno, oneside]{amsart}
\usepackage[foot]{amsaddr}

\usepackage{etoolbox}
\patchcmd{\section}{\normalfont}{\normalfont\Large}{}{}

\usepackage{amsmath,amssymb,amsthm}
\usepackage{mathtools}
\usepackage{amsfonts}
\usepackage{stmaryrd}
\usepackage{physics}
\usepackage[T1]{fontenc}
\usepackage[utf8]{inputenc}
\usepackage[english]{babel}

\usepackage[light]{kpfonts}
\usepackage{microtype}
\usepackage{bm}
\usepackage[autostyle]{csquotes}
\usepackage{graphicx} 

\usepackage{tikz}   
\usepackage{tikz-cd}  

\usepackage{enumitem}
\usepackage[bookmarksnumbered,linktocpage,hypertexnames=false,colorlinks=true,linkcolor=blue,urlcolor=black,citecolor=purple,anchorcolor=green,breaklinks=true]{hyperref}
\usepackage[doi=false, url=false, isbn=false, giveninits=true,date=year,style=alphabetic,clearlang=true,maxbibnames=99,backend=biber]{biblatex}
\renewbibmacro{in:}{}
\DeclareFieldFormat{postnote}{#1}
\DeclareFieldFormat{multipostnote}{#1}
\DeclareFieldFormat{pages}{#1}

\AtEveryBibitem{\clearfield{eprintclass}}
\AtEveryBibitem{\clearfield{url}}

\usepackage{tabularx}
\usepackage{cleveref}
\usepackage{xcolor}
\usepackage{multicol}
\usepackage{centernot}

\newtheorem{theoremintro}{Theorem}

\newtheorem{theorem}{Theorem}[subsection]
\newtheorem{lemma}[theorem]{Lemma}
\newtheorem{proposition}[theorem]{Proposition}
\newtheorem{corollary}[theorem]{Corollary}
\newtheorem{assumption}[theorem]{Assumption}

\theoremstyle{definition}
\newtheorem{definition}[theorem]{Definition}
\newtheorem{example}[theorem]{Example}
\newtheorem{remark}[theorem]{Remark}

\newcommand{\FF}{\mathbb F}

\def\NN{\mathbb{N}}
\def\RR{\mathbb{R}}
\def\QQ{\mathbb{Q}}
\def\ZZ{\mathbb{Z}}
\def\KK{\mathbb{K}}

\def\MM{\mathbb{M}}

\def\LL{\mathbb{L}}

\renewcommand{\epsilon}{\varepsilon}
\newcommand{\p}{\mathfrak p}
\newcommand{\m}{\mathfrak m}
\newcommand{\fF}{\mathfrak F}
\newcommand{\fK}{\mathfrak K}
\newcommand{\fM}{\mathfrak M}
\newcommand{\fH}{\mathfrak H}
\newcommand{\fR}{\mathfrak R}

\newcommand{\cC}{\mathcal{C}}

\newcommand{\cZ}{\mathcal{Z}}

\newcommand{\cH}{\mathcal{H}}
\newcommand{\cR}{\mathcal{R}}
\newcommand{\cO}{\mathcal{O}}

\newcommand{\cL}{\mathcal{L}}

\newcommand{\RV}{\mathrm{RV}}

\DeclareMathOperator{\definable}{def}
\DeclareMathOperator{\Def}{Def}
\DeclareMathOperator{\res}{res}
\DeclareMathOperator{\ac}{ac}
\DeclareMathOperator{\rv}{rv}
\DeclareMathOperator{\Trop}{Trop}

\DeclareMathOperator{\trop}{trop}
\DeclareMathOperator{\ftrop}{ftrop}
\DeclareMathOperator{\initial}{\mathrm{in}}

\DeclareMathOperator{\val}{val}
\DeclareMathOperator{\sval}{sval}

\DeclareMathOperator{\pos}{\mathcal{P}}

\DeclareMathOperator{\sign}{sign}
\DeclareMathOperator{\Spec}{Spec}
\DeclareMathOperator{\Sper}{Sper}
\DeclareMathOperator{\quot}{Quot}

\DeclareMathOperator{\closure}{cl}

\DeclareMathOperator{\supp}{supp}

\DeclareMathOperator{\sper}{Sper}

\setlist[enumerate,1]{label={(\roman*)},ref={\thetheorem (\roman*)}}
\def \arch{\mathrm{arch}}

\newcommand{\cl}[1]{\closure \left( {#1} \right)}

\DeclarePairedDelimiterX{\BW}[2]{\langle}{\rangle_{BW}}{#1, #2}

\usepackage{lineno}
 \usepackage[normalem]{ulem}

\usepackage[textsize=footnotesize, 
  linecolor=gray, bordercolor=gray, backgroundcolor=gray!25, 
  textcolor=black]{todonotes}

\usepackage{comment,soul,fullpage}

\makeatletter 
\def\arrowfill@#1#2#3#4{%
  $\m@th\thickmuskip0mu\medmuskip\thickmuskip\thinmuskip\thickmuskip
   \relax#4#1\mkern-4mu%
   \cleaders\hbox{$#4\mkern-2mu#2\mkern-2mu$}\hfill
   \mkern-4mu#3$%
}
\makeatother

\makeatletter
\DeclareFontEncoding{LS2}{}{\noaccents@}
\makeatother
\DeclareFontSubstitution{LS2}{libertinust1mathex}{m}{n}
\DeclareSymbolFont{libsymbols}{LS2}{libertinust1mathsym}{m}{n}
\DeclareMathSymbol{\lhook}{\mathrel}{libsymbols}{"41}
\DeclareMathSymbol{\rhook}{\mathrel}{libsymbols}{"42}

\title%
{A note on o-minimal tropicalizations
}

\author{Lorenzo~Baldi$^\dagger$}
\address{$\dagger$ MPI of Molecular Cell Biology and Genetics \& Center for Systems Biology, Dresden}
\author{Máté~L.~Telek$^\ddagger$}
\address{$\ddagger$ Budapest University of Technology and Economics}
\date{June 2026}

\begin{document}

\begin{abstract}
In this note, we provide several equivalent descriptions of the tropicalization of definable sets in a polynomially bounded o-minimal expansion of a real closed field. We show that the tropicalization can be described in multiple ways: as the image under the valuation map, as the image under the tropicalization map of the archimedean points in the o-minimal spectrum, in terms of initial degenerations of the definable set and initial nonnegativity cones, or via the real tropicalization of polynomials belonging to the corresponding nonnegativity cone. Moreover, we prove analogous results for the fine tropicalization of definable sets, where the valuation map is replaced by the RV sort map, which records not only the valuation of a point but also its angular component.
\end{abstract}

\maketitle

\setcounter{tocdepth}{1}

\vspace{-0.8cm}
\section{Introduction}

The origins of tropical geometry can be traced back to the work of Bergman \cite{bergmanLogarithmicLimitSetAlgebraic1971} and Viro \cite{viroRealAlgebraicVarieties1983}. Since their pioneering work, the theory has developed in several directions and has found applications in both pure and applied mathematics. Tropical geometry studies algebraic varieties (or more generally, definable sets) via their degenerations, often referred to as their \emph{polyhedral shadows}. These polyhedral shadows are typically easier to work with than their algebraic counterparts, while still retaining significant information about the original objects. For example, they have been used to compute Gromov–Witten invariants \cite{mikhalkinEnumerativeTropicalAlgebraic2005}, study toric compactifications of algebraic varieties \cite{tevelevCompactificationsSubvarietiesTori2007b}, and construct real algebraic varieties with prescribed topology \cite{viroSixteenthHilbertProblem2008,brugalleCombinatorialPatchworkingBack2024,Rau2025}. On the applied side, tropical geometry has been used to determine the generic number of solutions of parametrized polynomial systems \cite{helminckGenericRootCounts2025,feliuRootBoundsVertical2026}, solve such systems via homotopy continuation methods \cite{huberPolyhedralMethodSolving1995,leykinPolyhedralHomotopies2019, helminckTropicalMethodSolving2024}, and compute Feynman integrals \cite{borinskyTropicalFeynmanIntegration2023}.

One of the ways to define the tropicalization $\Trop(X)$ of a set $X\subset(\FF \setminus \{0\})^n$, where $\FF$ is a valued field, is as the closure, in the euclidean topology of $\RR^n$, of the image of $X$ under the coordinate-wise valuation map. 
While this definition is straightforward, it is convenient to have other equivalent descriptions of the points in $\Trop(X)$, that are often better suited for understanding how properties of $X$ translate to properties of $\Trop(X)$, and vice versa. When $\FF$ is algebraically closed and $X$ is an algebraic set, the Fundamental Theorem of Tropical Algebraic Geometry gives these equivalent descriptions for $\Trop(X)$, which can be characterized in terms of the initial degenerations of $X$, of the Berkovich analytification of $X$, or of the tropicalizations of all polynomials vanishing on $X$ (see \cite[Th.~3.2.3]{MaclaganSturmfels}, \cite[Th.~4.2]{Draisma}, and \cite{EinsiedlerKapranovLind}).
These results rely on the algebraically closedness of the ground field, but can be adapted to give interesting information also for algebraic sets defined over a real closed field, see e.g. \cite{speyerTropicalTotallyPositive2005,brandenburgTropicalPositivityDeterminantal2023}.

However, this approach is not sufficient if one wants to take into account the \emph{order} of the real closed field, which makes possible to define \emph{semialgebraic} sets (or more generally, definable sets in an o-minimal structure) in addition to algebraic sets. This necessity led to the works \cite{develinTropicalPolytopesCellular2007} (for the special case of polyhedra) and \cite{Alessandrini}, where the tropicalization of semialgebraic sets is systematically studied for the first time.
Since then, this topic has attracted an increasing attention, see e.g. \cite{Allamigeon2020,Telek,blekhermanMomentsSumsSquares2025}. In particular, in their seminal paper, Jell, Scheiderer, and Yu established an analogue of the Fundamental Theorem in the semialgebraic setting \cite[Th.~6.9]{jellRealTropicalizationAnalytification2020}.

In this note, we show that this theorem can be generalized to definable sets in polynomially bounded o-minimal structures. While retaining the tameness properties of semialgebraic sets, definable sets in polynomially bounded o-minimal structures are considerably more general, as they include for instance globally subanalytic sets, see e.g. \cite{driesGeometricCategoriesOminimal1996b}, and allows the use of \emph{real} exponents for monomial functions, with potential applications e.g. to the study of irrational toric varieties \cite{postinghelDegenerationsRealIrrational2015}.

Furthermore, in our fundamental theorem we prove a characterization of the tropicalization of a definable set in terms of \emph{real} initial degenerations (see \Cref{rem:initial} for a comparison with the classical ones) , which is new already in the semialgebraic case. After the following theorem, which is our main result, we discuss the objects which appear in the statement, and compare the results obtained with the existing literature. 
\begin{theoremintro}[Fundamental Theorem of Tropical O-minimal Geometry, see Theorem  \ref{thm:fundamental_theorem}]
\label{Thm:A}
    Let $\fK \prec \fF \prec \fM$ be elementary extensions of polynomially bounded o-minimal expansions of the real closed fields $\KK \subset \FF \subset \MM$. Furthermore, assume that $\KK \subset \FF \subset \MM$ is an extension of valued real closed fields such that the valuation on $\FF$ is nontrivial and $\val(\MM^\times) = \RR$. For any definable set $X \subset \KK_{>0}^n$, the following subsets of $\RR^n$ are equal.
    \begin{enumerate}
    \item $\Trop(X)$;
       \item $\trop ( \widetilde{X_\fF}^\arch)$;
          \item  $\overline{\{ \, w \in \Gamma^n_\mathbb{F} \mid \initial_w(X_\fF) \neq \emptyset \, \}}$;
        \item $\val(X_\fM)$;
        \item $\bigcap_{f \in \mathcal{P}(X)} \{ \, w \in \RR^n \mid \trop_{r}(f)(1,w) \geq 0 \, \}$.
    \end{enumerate}
    Moreover, the intersection in (v) can be taken to be finite.
\end{theoremintro}
As observed e.g. in \cite[Sec.~6]{jellRealTropicalizationAnalytification2020}, one must consider the intersection of a set $X$ with each orthant in $\KK^n$ separately, in order to study its full tropicalization. The gluing of the tropicalizations of the intersections of $X$ with different orthants leads to the definition of \emph{real phase structures} in the algebraic case, see e.g. \cite{Rau2025}. To simplify our notation, we restrict our attention to the positive orthant $\KK_{>0}^n$ only.

In (ii), the archimedean points in the o-minimal spectrum $\widetilde{X_{\fF}}$ play the role of the real Berkovich analytification of $X$, cf. \cite[Th.~3.17]{jellRealTropicalizationAnalytification2020} and \Cref{rem:berkovic}. In our work, they are constructed as special points in the o-minimal spectrum of the extension $X_\fF$ of $X$ to $\fF$, see \Cref{Sec:RelArchPoints}. Those points admit a natural valuation extending the valuation of $\FF$, see \Cref{lem:valuation_archimedean}, which is used to define the set in point (ii) in Theorem~\ref{Thm:A}.

Part (iii) characterizes the points of $\Trop(X)$ in terms of real initial degenerations, in analogy to the case of algebraic varieties. The initial degenerations $\initial_w(X_\fF)$ are closed definable subsets in the residue field of $\FF$, see \Cref{prop:ValImageClosedM}. Furthermore, we show that $\initial_w(X_{\fF})$ can be described by the initial forms of polynomials contained in the nonnegativity cone $\mathcal{P}(X)$ of $X$ (Proposition~\ref{Prop:InitialDegViaPreorder}). Thus, in the definable setting, initial nonnegativity cones play a role analogous to that of initial ideals in algebraic geometry. 

Part (v) further strengthens this analogy. There, we consider the real tropicalizations of polynomials that are nonnegative on $X$. These real tropical polynomials have coefficients in the real tropical hyperfield, which record both the sign and the valuation of the coefficients of the original polynomials. The idea of using such real tropicalizations was introduced in \cite{jellRealTropicalizationAnalytification2020}. Indeed, in the special case of semialgebraic sets, the equivalence of conditions (i), (ii), (iv), and (v) was already established in \cite[Th.~6.9]{jellRealTropicalizationAnalytification2020}.

Our techniques for the proof of \Cref{Thm:A} combine the ideas from \cite{jellRealTropicalizationAnalytification2020} with results in model theory and o-minimal geometry. The study in \cite{driesTConvexityTameExtensions1995,driesTConvexityTameExtensions1997} of valued real closed fields is particularly important, together with the o-minimal spectrum, for which we refer to the many references in \Cref{Sec:OminimalSpectrum}. The o-minimal spectrum plays the role of the real spectrum in semialgebraic geometry, see \cite{carralNormalSpectralSpaces1983,tresslRealSpectrumContinuous1997}.

\medskip

Another useful ingredient for the proof of \Cref{Thm:A} (and for understanding tropicalizations of definable sets in general) is to enrich the tropicalization map to retain information beyond the valuation. This idea has already appeared in several forms in the literature. For example, the phase tropicalization map records not only the valuation of an element but also its phase \cite{mikhalkinDecompositionPairsofpantsComplex2004,mikhalkinEnumerativeTropicalAlgebraic2005,viroBasicConceptsTropical2011,kimNaturalTopologicalManifold2021}.
In the real setting, the phase reduces to the sign of an element, thus it does not provide any additional information, since we restrict our attention to the positive orthant.
To obtain a refinement of the tropicalization, we instead follow \cite{maxwellGeometryTropicalExtensions2024} and enrich the valuation by recording the full angular components. We rephrase the ideas of \cite{maxwellGeometryTropicalExtensions2024} to the real closed setting using the RV sort structure in model theory, see e.g. \cite{flennerRelativeDecidabilityDefinability2011}.

We define the \emph{fine tropicalization} of a definable set as its image under the coordinate-wise Residual Value sort map $\rv\colon \FF^\times \to \FF^\times/(1+\mathfrak{m}_\FF)$. Here, $\mathfrak{m}_\FF$ denotes the maximal ideal of the valuation ring of~$\FF$. If we fix a splitting of $\FF$, there is an associated notion of \emph{angular component}, denoted $\ac$, and $\rv \cong \ac \times \val$ (see \Cref{Sec:RealClosedValuedFields}).
In this note, we prove an analogue of the Fundamental Theorem in the fine tropical setting. The equality of (iii) and (iv) below in \Cref{Thm:B} is a real counterpart of \cite[Theorem 5.2]{maxwellGeometryTropicalExtensions2024}. To make the analogy more explicit, we state here the theorem in the case $\FF = \RR((t^\RR))$, the field of Hahn series with real exponents. Moreover, we describe the fine tropicalization in terms of archimedean points of the o-minimal spectrum, and by using the $\RV$-tropicalization of nonnegative polynomials, whose coefficients live in the $\RV$-tropical hyperfield.

\begin{theoremintro}[Fundamental Theorem of Fine Tropical O-minimal Geometry, see \Cref{thm:fine_fundamental}]
\label{Thm:B}
$\qquad \qquad \qquad$ Let $\FF = \RR((t^\RR))$ be the field of Hahn series and let $\fF$ be a polynomially bounded o-minimal expansion of $\FF$.
Then, for any closed definable subset $X \subset \FF_{>0}^n$, the following sets are equal: 
    \begin{enumerate}
        \item $\ftrop ( \widetilde{X_\fF}^\arch)$;
        \item $\bigcap_{f \in \mathcal{P}(X_\fF)}     \{ z \in \RR_{>0}^n \times \RR^n \mid \trop_{\rv}(f)(z) \geq 0\}$;
        \item $\rv(X_\fF)$;
        \item  $\bigcup_{w \in \RR^n} \initial_w(X_\fF) \times \{ w \}$.
    \end{enumerate}
\end{theoremintro}

The paper is organized as follows. In Section~\ref{Sec:Prelim}, we recall some basic notions and set up the notation for valued real closed fields, hyperfields, and model theory. In Section~\ref{Sec:TropDefSets}, we begin our investigation of the tropicalization of definable sets. Section~\ref{Section:TropofDefFunctions} and~\ref{Sec:InitialDegen} are devoted to the tropicalizations and initial forms of polynomials, and to their relationship with initial degenerations. In Section~\ref{Sec:FristVersion}, we firsty prove weak versions of Theorems~\ref{Thm:A}, \ref{Thm:B}, without considering the o-minimal spectrum and the closure in \Cref{Thm:A}, (iii).

In Section~\ref{Sec:Sec4}, we turn our attention to the closure and relate the points in the closure to the archimedean points of the o-minimal spectrum of the definable set. After discussing several characterizations of the o-minimal spectrum in Section~\ref{Sec:OminimalSpectrum}, we show that the residue fields corresponding to archimedean points admit a unique valuation, which we will use to define the tropicalization and $\rv$-tropicalization maps. This is the content of Sections~\ref{Sec:RelArchPoints} and~\ref{Sec:RelArchPointsVal}. In Section~\ref{sec:strong_density}, we prove the Strong Density Theorem (\Cref{thm:strong_density}) and other key technical ingredients for the proofs of Theorems~\ref{Thm:A} and~\ref{Thm:B}, which is concluded in Section~\ref{sec:fundamental_theorems}.

\section{Preliminaries}
\label{Sec:Prelim}

To provide a concise overview of the key mathematical objects considered in this work and their associated notation, we collect them in Table~\ref{Table}. In this section, we elaborate on these objects and recall the necessary background on valued real closed fields, model theory, and hyperfields.

\begin{table}[h]
\centering
\caption{
{\bf Table of notation}}
\begin{tabular}{c|l}
\hline
$\qquad R^\times \qquad $ & \small set of invertible elements of a ring $R$ \\ \hline
$\qquad [n] \qquad $ & \small $\{1,\dots,n\}$ for $n \in \mathbb{N}$ \\ \hline
$\qquad \KK \qquad $ & \small valued real closed field \\ \hline
$\qquad \FF \qquad $ & \small nontrivially valued real closed field \\ \hline
$\qquad \MM \qquad $ & \small nontrivially valued real closed field with value group $\RR$ \\ \hline
$\qquad \Gamma_\FF \qquad $ & \small value group \\ \hline
$\qquad \mathcal{O}_\FF \qquad $ & \small valuation ring  \\ \hline
$\qquad \mathfrak{m}_\FF \qquad $ & \small valuation ideal \\ \hline
$\qquad k_\FF \qquad $ & \small residue field \\ \hline
$\qquad \RV_\FF \qquad $ & \small Residual Value sort hyperfield \\
\hline
$\qquad \val_\FF \colon \FF \to \RR \cup\{\infty\} \qquad $ & \small valuation map \\ \hline
$\qquad\res_\FF\colon \mathcal{O}_\FF \to  k_\FF \qquad $ & \small residue map \\ \hline
$\qquad\ac_\FF\colon \FF^\times \to  k_\FF^\times \qquad $ & \small angular component map \\ \hline
$\qquad \rv_\FF \colon \FF \to \RV_\FF  \qquad $ & \small residual value map \\ \hline
$ \qquad \fK, \fF, \fM \qquad$ & \small o-minimal expansions of a real closed field $\KK, \FF, \MM$ resp. \\ \hline
$\qquad \Lambda = \Lambda(\fK) \qquad$ & \small the field of exponents of $\fK$ \\
\hline
$\qquad R[\vb x^\Lambda] \qquad$ & \small set of polynomials with coefficients in a ring $R$ or a hyperfield $R$ and exponents in $\Lambda$ \\
\hline
$\qquad X_\fF \qquad $& \small $\fF$-extension of a definable set $X$ in $\fK$\\ \hline
$\qquad \widetilde{X} \rule{0pt}{2.4ex} \qquad $& \small o-minimal spectrum of a definable set $X$\\
\hline
\end{tabular}
\label{Table}
\end{table}

\subsection{Real closed fields and valuations}
\label{Sec:RealClosedValuedFields}
A \emph{real field} is a field $\mathbb{K}$ that admits an \emph{ordering} $\leq$, that is, a total order $\leq$ on $\mathbb{K}$ that is compatible with the field operations: for all $a,b,c \in \mathbb{K}$, we have
\begin{align*}
    a \leq b \Rightarrow a + c \leq b + c, \quad \text{ and } \quad a \leq b,\, 0 \leq c \Rightarrow ac \leq bc.
\end{align*}
For an \emph{ordered field} $(\mathbb{K},\leq)$, we write $\mathbb{K}_{>0}$ for the set of elements $a \in \mathbb{K}$ with $a > 0$, and define an \emph{absolute value} as $\abs{a} \coloneqq \max\{a,-a\}$.
A field $\mathbb{K}$ is \emph{real closed} if it admits a unique ordering and every odd degree polynomial with coefficients in $\mathbb{K}$ has a root in $\mathbb{K}$. We refer the reader to \cite[Sec.~2.2]{BasuPollackRoy} or \cite[Sec.~1.2]{Scheiderer2024} for other equivalent characterizations. The inclusion $\mathbb{K} \subset \mathbb{F}$ is an \emph{extension of real closed fields} if $\KK$ and $\FF$ are real closed fields and the ordering on $\mathbb{F}$, when restricted to $\mathbb{K}$, coincides with the ordering on $\mathbb{K}$.

In tropical geometry, we are interested in fields $\mathbb{F}$ equipped with a \emph{valuation map} $\val\colon \mathbb{F} \to \mathbb{R} \cup \{\infty\}$, i.e. a map satisfying the following conditions:
\begin{itemize}
    \item[(1)] $\val(a) = \infty$ if and only if $a = 0$;
    \item[(2)] $\val(ab) = \val(a) + \val(a)$; and
    \item[(3)] $\val(a+b) \geq \min\{ \, \val(a), \val(b)\, \}$ for all $a,b \in \mathbb{F}$.
\end{itemize}
Condition (3) is called the \emph{ultrametric inequality}.
We denote by $\Gamma_{\FF} \coloneqq \val(\mathbb{F}^\times)$ the \emph{value group} of $\mathbb{F}$, which is an additive subgroup of $(\mathbb{R},+)$.
If $\Gamma_{\FF}  \neq \{0\}$, the valuation is called \emph{nontrivial}. The \emph{valuation ring} and the \emph{valuation ideal} are defined as
\begin{align*}
    \mathcal{O}_{\mathbb{F}} \coloneqq \{ \, a \in \mathbb{F} \mid \val(a) \geq 0 \,\}, \quad \text{and} \quad \mathfrak{m}_{\mathbb{F}} \coloneqq \{ \, a \in \mathbb{F} \mid \val(a) > 0 \,\}
\end{align*}
respectively. The valuation ring is a local ring with maximal ideal $\mathfrak{m}_{\mathbb{F}}$ (see e.g. \cite[Proposition 2.4.5]{Knebusch2022}). We write $k_{\mathbb{F}} \coloneqq \mathcal{O}_{\mathbb{F}}/\mathfrak{m}_\mathbb{F}$ for the residue field and $\res_{\mathbb{F}}\colon \mathcal{O}_\mathbb{F} \to k_\mathbb{F}$ for the natural projection map.

We call $\mathbb{F}$ a \emph{valued real closed field}, if the valuation is \emph{compatible with the ordering} of $\mathbb{F}$, that is, if for all $a,b \in \mathbb{F}$ we have
\begin{align}
\label{Eq:OrderCompatibleValuation}
0 \leq a \leq b \Rightarrow \val(a) \geq \val(b).
\end{align}
This condition is satisfied if and only if the valuation ring $\cO_\FF$ is a \emph{convex subring} of $\FF$. This means that, if $a\le c \le b$ and $a, b \in \cO_\FF$, then $c\in \cO_\FF$.
In this case, the residue field $k_\mathbb{F}$ is also a real closed field whose ordering is given by $\res_{\mathbb{F}}(a) \geq 0$ whenever $a \in \mathcal{O}^\times_{\mathbb{F}}$ and $a \geq 0$ \cite[Prop.~3.5.10]{Scheiderer2024}. An \emph{extension of valued real closed fields} $\mathbb{K} \subset \mathbb{F}$ is an extension of real closed fields such that the valuation on $\mathbb{F}$ restricts to the valuation on $\mathbb{K}$. We refer the reader to \cite{Knebusch2022,Scheiderer2024} for more details on order compatible valuations and convex subrings.

\begin{example}
\label{ex:PuiseuxSeries}
Our primary example for a real closed field with a nontrivial order compatible valuation is the field of \emph{Hahn series} with real coefficients and exponents, denoted $\mathbb{R}((t^\RR))$, whose nonzero elements are formal series of the form
\begin{align}
\label{Eq:aPuiseuxSeries}
a = \sum_{v \in \supp(a)} a_v t^{v}, \qquad a_v \in \RR \setminus \{0\}
\end{align}
where $\supp(a) \subset \RR$ is such that there exists $v_0 \in \supp(a)$ that equals to the infimum of $\supp(a)$ in $\RR$. Then the valuation of $a$ is given by $\val(a) \coloneqq v_0$, the value group of $\mathbb{R}((t^\RR))$ is $\mathbb{R}$  and 
the residue field $k_{\mathbb{R}((t^\RR))}$ is isomorphic to $\mathbb{R}$.
An element $a$ as in \eqref{Eq:aPuiseuxSeries} is \emph{nonnegative}, written as $a \geq 0$, if its leading coefficient $a_{v_0}$ is positive in the ordering of~$\mathbb{R}$. We refer to \cite[Ch.~6]{allingFoundationsAnalysisSurreal1987} for more details on Hahn series.

The field of Hahn series $\RR((t^\RR))$ is a valued real closed extension of $\RR$, where $\RR$ is equipped with the trivial valuation, i.e., $\val(a)=0$ for all $a\in \RR$. Although this case might seem "trivial", the study of tropicalizations over trivially valued real closed fields leads to interesting results and is worthwhile in its own right.
\end{example}

\begin{assumption}
    \label{assumtion:valued_fields}
     Let $\KK \subset \FF \subset \MM$ be an extension of valued real closed fields such that the valuation on $\FF$ is nontrivial and $\Gamma_\MM = \RR$. Every nontrivially valued real closed field admits a \emph{splitting} or \emph{cross section}, that is, there exitsts a group homomorphism $\psi\colon (\Gamma_\FF,+) \to (\FF^\times,\cdot)$ such that $\val(\psi(w)) = w$ for all $w\in \Gamma_\FF$, see e.g. \cite[Prop.~3.5.10]{Scheiderer2024}. We denote the image of $w \in \Gamma_\FF$ under the splitting $\psi$ by $t^w$. We write $\ac_\FF \colon \FF^\times \to k^\times_\FF$, $a \mapsto \res_\FF(t^{-\val(a)}a)$
     for the associated \emph{angular component} map, which depends on the choice of the splitting.
     
     If $\val$ is a valuation, then $\lambda \cdot \val$ is a valuation as well for any $\lambda \in \RR_{>0}$. We will then assume any nontrivial valuation $\val$ to be normalized so that the value group contains $1$, and define $t \coloneqq t^1$.
\end{assumption}

In the case of Hahn series (Example~\ref{ex:PuiseuxSeries}), the angular component remembers the coefficient of the term of lowest degree in the series. To simultaneously remember both the valuation and angular component, we use the following construction, which is commonly used in the model-theoretic study of valued fields, see e.g. \cite{flennerRelativeDecidabilityDefinability2011, yinGeneralizedEulerCharacteristic2017} and references therein. Let $\FF$ be a nontrivially valued real closed field, and consider the  multiplicative subgroup $1 + \m_\FF$ of $\mathbb{F}^\times$. 
The quotient multiplicative group $\RV^\times_\FF \coloneqq \FF^\times / (1 + \m_\FF)$ is called the \emph{Residual Value sort} or \emph{RV sort} of $\FF$, and we denote by $\rv\colon \FF^\times \to \RV^\times_\FF$ the quotient map. For later reference, we summarize the main objects in the following commutative diagram
\begin{equation}
\label{Eq:CommutativeDiagram}
\begin{tikzcd}[sep=large]
& \cO_\FF^\times \arrow[r, hookrightarrow] \arrow[d, "\res_\FF"]
& \FF^\times \arrow[d, "\rv"] \arrow[dr, "\val"] \\
1 \arrow[r] & k^\times_\FF \arrow[r] & \RV^\times_\FF \arrow[r, "\val_{\rv}"] & \Gamma_\FF \arrow[r] & 0
\end{tikzcd}
\end{equation}
where the bottom row is exact. A choice of a splitting gives a (non-canonical) isomorphism 
\begin{alignat}{2}
\label{Eq:RVIso}
    \RV^\times_\FF&\longleftrightarrow && k^\times_\FF \times \Gamma_\FF \nonumber \\
    z = \rv(a) & \, \, \longmapsto &&(\,\ac_\FF (a), \,\val(a)\,) \\
    \rv(ut^w)  & \, \, \longmapsfrom &&(\,\res_\FF(u),\,w\,) \nonumber.
\end{alignat}
We fix such an isomorphism and identify $\RV^\times_\FF$ with $k^\times_\FF \times \Gamma_\FF$. Now, the maps in~\eqref{Eq:CommutativeDiagram} can be described as $k^\times_\FF \to \RV_\FF^\times$, $y \mapsto (y,0)$ and $\val_{\rv}\colon \RV_\FF^\times \to \Gamma_\FF$, $(y,w) \mapsto w$.

We add an additional symbol to $\RV^\times_\FF$ and define $\RV_\FF \coloneqq \RV^\times_\FF \cup \{0\}$. We then extend the $\rv$ map to $\rv\colon \FF \to \RV_\FF$ by setting $\rv(0) = 0$.
The set $\RV_\FF$ has a structure of a hyperfield, as we will see in Section~\ref{Sec:Hyperfields}.
Moreover, $\RV_\FF$ has an ordering: $z_1 \prec z_2$ in $\RV_\FF$ if and only if $\rv^{-1}(z_1) < \rv^{-1}(z_2)$. We write
\[ \RV_{\FF,>0} \coloneqq \{ \, \rv(a) \in \RV^\times_\FF \mid a > 0\, \} \cong k_{\FF,>0} \times \Gamma_\FF  \]
for the set of positive elements in $\RV_\FF$. For $(y_1, w_1), (y_2,w_2) \in \RV_{\FF,>0}$, the order is explicitly given as 
\begin{align*}
    (y_1, w_1) \prec  (y_2,w_2) \Leftrightarrow \begin{cases}
         w_1 > w_2, \text{ or } \\
         w_1 = w_2 \text{ and } y_1 < y_2
    \end{cases}.
\end{align*}

\subsection{Hyperfields}
\label{Sec:Hyperfields} 
Hyperfields are generalizations of fields, where addition can be a multivalued map. Their study goes back at least to the work of Krasner \cite{krasnerLapproximationCorpsValues1959,krasnerClassHyperringsHyperfields1983}, and their first appearance in tropical geometry occurred in works of Viro \cite{viroHyperfieldsTropicalGeometry2010, viroBasicConceptsTropical2011}. In this section, we recall some background on hyperfields that is necessary for Section~\ref{Section:TropofDefFunctions}. Our exposition closely follows \cite{maxwellGeometryTropicalExtensions2024}.

\begin{definition}
A \emph{hyperfield} is a tuple $(\mathbb{H}, \oplus, \odot, 0,1)$ where $\mathbb{H}^\times \coloneqq \mathbb{H} \setminus \{0\}$, $(\mathbb{H}^\times, \odot, 1)$ is an abelian group, and
\[
\oplus \colon \mathbb{H} \times \mathbb{H} \rightarrow 2^{\mathbb{H}} \setminus \{\emptyset \} \] 
satisfies the following axioms:
\begin{itemize}
\item Identity:  $0 \oplus x =\{x\}$ for all $x \in \mathbb{H}$.
\item Unique Inverse: For all $x \in \mathbb{H}$, there is a unique $-x \in \mathbb{H}$ such that $0 \in x\oplus (-x)$.
\item Reversibility: For all $x,y,z \in \mathbb{H}$, we have $z \in x\oplus y$ if and only if $y \in z\oplus(-x)$.
\item Commutativity: $x\oplus y=y \oplus x$ for all $x,y \in \mathbb{H}$,
\item Associativity: $(x\oplus y) \oplus z=x\oplus (y\oplus z)$, where $(x\oplus y) \oplus z = \bigcup_{u \in x\oplus y} u \oplus z$ and  $x\oplus (y \oplus z) = \bigcup_{u \in y\oplus z} x \oplus u$.
\item Distributivity: $z \odot(x\oplus y)=z \odot x\oplus z \odot y$, where $z \odot(x\oplus y) = \bigcup_{u \in x\oplus y} z \odot u$.
\item $0 \odot x = 0$ for all $x \in H$.
\end{itemize}
\end{definition}

The simplest instance of hyperfields are ordinary fields. In this case, the hyperaddition $x\oplus y$ always gives a singleton. In the next example, we present two additional examples that play an important role in tropical geometry.
 \begin{example}
     \begin{itemize}
     \item[]
         \item[(a)]The \emph{Krasner hyperfield}  is the set $\mathbb{H}_K \coloneqq \{0,1\}$ with multiplication $1 \odot 1 = 1, \, 0 \odot 1 = 1 \odot 0 = 0 \odot 0 = 0$ and hyperaddition defined as
         \[ 0\oplus 0 = \{0\}, \qquad 0\oplus 1 = 1 \oplus 0 = \{1\}, \qquad 1 \oplus 1 = \{0,1\}. \]
         \item[(b)] The \emph{sign hyperfield} is the set $\mathbb{S} \coloneqq \{-1,0,1\}$ with multiplicative group $\mathbb{S}^\times = \{-1,1\}$ and hyperaddition defined~as
         \[0 \oplus x = \{x\}, \qquad x\oplus x = \{x\},  \quad \text{ for all } x \in \mathbb{S} \qquad \text{ and } \quad 1 \oplus -1 = \{-1,0,1\}.\]
     \end{itemize}
 \end{example}
 Following \cite[Def.~2.6]{maxwellGeometryTropicalExtensions2024}, we define the \emph{tropical extension of a hyperfield} as follows.
 
 \begin{definition}
 \label{Def:HyperAddition}
 Let $\mathbb{H}$ be a hyperfield and $(\Gamma,+)$ an (additive) ordered abelian group. The \emph{tropical extension of $\mathbb{H}$ by $\Gamma$} is the set
\[\mathbb{H} \rtimes \Gamma=\left\{\, (x, w) \mid x \in \mathbb{H}^{\times}, w \in \Gamma\,\right\} \cup\{\,0\,\}\]
where multiplication is given by $\left(x_1, w_1\right) \odot\left(x_2, w_2\right)\coloneqq\left(x_1 \odot_{\mathbb{H}} x_2, w_1+w_2\right)$ and addition is given by
\[
\left(x_1, w_1\right) \oplus \left(x_2, w_2\right) \coloneqq \begin{cases}\left(x_1, w_1\right) & w_1<w_2, \\
\left(x_2, w_2\right) & w_2<w_1, \\ 
\left\{\,(x, w_1) \mid x \in x_1 \oplus_{\mathbb{H}} x_2\,\right\} & w_1=w_2,\; 0 \notin x_1 \oplus x_2,\\ 
\left\{\, (x, w_1) \mid x \in\left(x_1 \oplus_{\mathbb{H}} x_2\right) \setminus \{0\} \, \right\} \cup  \left\{\,(y, v) \mid v>w_1, y \in \mathbb{H}^{\times}\,\right\} \cup\{0\} & w_1=w_2, \; 0 \in x_1 \oplus x_2.\end{cases}
\]
 \end{definition}

Tropical extensions of hyperfields are again hyperfields. If $\mathbb{H}$ is the Krasner hyperfield $\mathbb{H}_K$, the sign hyperfield $\mathbb{S}$ or a field, then the tropical extension $\mathbb{H} \rtimes \Gamma$ is a \emph{stringent} hyperfield, that is, the hyperaddition $x \oplus y$ is multivalued if and only if $0 \in x \oplus y$ \cite[Sec.~4]{bowlerClassificationDoublyDistributive2021}. These stringent tropical extensions are exactly the hyperfields of primary interest in this paper. For the reader’s convenience, we now describe their construction explicitly.
 \begin{example}
 \label{Examples:Hyperadditions}
 Let $\FF$ be a nontrivially valued real closed field with value group $\Gamma_\FF$ and residue field $k_\FF$.
     \begin{itemize}
         \item[(a)]The \emph{tropical hyperfield} $\mathbb{T}_\FF \coloneqq \mathbb{H}_K \rtimes \Gamma_\FF$ is the extension of the Krasner hyperfield $\mathbb{H}_K$ by $\Gamma_\FF$. 
         We denote by $\infty$ the identity element in $\mathbb{T}_\FF$. Since $\mathbb{H}^\times = \{ 1\}$, we identify $\mathbb{T}_\FF$ with $\Gamma_\FF \cup \{ \infty \}$.
         The hyperaddition in $\mathbb{T}_\FF$ is given by
\[
w_1 \oplus w_2 = \begin{cases}\min(w_1, w_2)  & w_1 \neq w_2, \\  \left\{\,v \in \Gamma_\FF \mid v \geq w_1 \,\right\} \cup\{\infty\} & w_1=w_2.\end{cases}
\]
         \item[(b)] The \emph{signed tropical hyperfield} $\mathbb{RT}_\FF \coloneqq \mathbb{S} \rtimes \Gamma_\FF$ is the extension of the sign hyperfield $\mathbb{S}$ by $\Gamma_\FF$. Its identity element is again denoted by $\infty$, and the hyperaddition is given by
         \[
\left(x_1, w_1\right) \oplus \left(x_2, w_2\right) \coloneqq \begin{cases}\left(x_1, w_1\right) & w_1<w_2, \\
\left(x_2, w_2\right) & w_2<w_1, \\ 
(x_1, w_1) & w_1=w_2,\; x_1 = x_2,\\ 
\left\{\, (\pm 1, v) \mid v\in \Gamma_\FF, \, v\geq w_1 \,\right\} \cup\{\infty\} & w_1=w_2, \; x_1 \neq x_2.\end{cases}
\]
        \item[(c)] The \emph{RV-tropical hyperfield} $\RV_\FF \coloneqq k_\FF \rtimes \Gamma_\FF$ is the extension of $k_\FF$ by $\Gamma_\FF$.
        \[
\left(x_1, w_1\right) \oplus \left(x_2, w_2\right) \coloneqq \begin{cases}\left(x_1, w_1\right) & w_1<w_2, \\
\left(x_2, w_2\right) & w_2<w_1, \\ 
(x_1 + x_2, w_1)  & w_1=w_2,\; x_1 + x_2 \neq 0,\\ 
 \left\{\,(y, v) \mid v\in \Gamma_\FF,\, v>w_1,\, y \in k^{\times}_\FF\,\right\} \cup\{0\} & w_1=w_2, \; x_1 + x_2 = 0.\end{cases}
\]
     \end{itemize}
     
 \end{example}
The motivation behind the definition of hyperaddition in a tropical extension is to capture the cancellation of leading terms in $\FF$. For example, in each case of \Cref{Examples:Hyperadditions}, the hypersum $(x_1,w_1) \oplus (x_2,w_2)$ is a singleton if and only if, for every choice of preimages $a_1,a_2 \in \FF$, one can guarantee that no cancellation of leading terms occurs in $a_1+a_2$ based on the data encoded by $(x_1,w_1)$ and $(x_2,w_2)$. This happens, for instance, when the valuations $w_1$ and $w_2$ are different, when the leading coefficients have the same sign, or, in the case of the RV-tropical hyperfield, when the leading coefficients do not cancel.
 
The following properties of hyperfields will be used in Section~\ref{Section:TropofDefFunctions}.
A map $\varphi\colon \mathbb{H}_1 \to\mathbb{H}_2$ is a \emph{hyperfield morphism} if its restriction $\varphi\colon \mathbb{H}_1^\times \to\mathbb{H}_2^\times$ is a group homomorphism, $\varphi\left(0\right) =0$, and we have
\begin{align*}
\varphi\left(x \oplus y\right) \subset \varphi(x) \oplus \varphi(y).
\end{align*}

 \begin{lemma}
 \label{Lemma:HyperfieldMorphisms}
 Let $\mathbb{F}$ be a nontrivially valued real closed field.
    We have the following commutative diagram of morphisms of hyperfields
\begin{equation}
\label{Eq:HyperfieldMaps}
    \begin{tikzcd}[sep=large]
& \FF \arrow[dl, "\rv"] \arrow[d, "\sval"] \arrow[dr, "\val"] & \\
\RV_\FF \arrow[r, "\phi"] & \mathbb{RT}_\FF \arrow[r, "\psi"] & \mathbb{T}_\FF,
\end{tikzcd}
\end{equation}
where the maps are given as follows    
\begin{alignat*}{5}
\rv\colon& \FF \longrightarrow \RV_\FF 
&\quad \sval\colon & \FF \longrightarrow \mathbb{RT}_\FF 
&\quad \val\colon & \FF \longrightarrow \mathbb{T}_\FF
&\quad \phi\colon& \RV_\FF \longrightarrow \mathbb{RT}_\FF 
&\quad \psi\colon& \mathbb{RT}_\FF \longrightarrow \mathbb{T}_\FF \\
& a \mapsto (\ac_\FF(a),\val(a)) 
&\quad & a \mapsto (\sign(a),\val(a))  
&\quad & a \mapsto \val(a)
&\quad & (y,w) \mapsto (\sign(y),w) 
&\quad & (x,w) \mapsto w \\
& 0 \mapsto 0 
&\quad & 0 \mapsto \infty  
&\quad & 0 \mapsto \infty 
&\quad & 0 \mapsto \infty 
&\quad & \infty \mapsto \infty
\end{alignat*}
    \end{lemma}

    \begin{proof}The commutativity of the diagram \eqref{Eq:HyperfieldMaps} is clear by construction, noting that $\sign(\ac_\FF(a)) = \sign(a)$ for each $a \in \FF$. Since the sign, angular component and valuation maps are group homomorphism, all the maps in \eqref{Eq:HyperfieldMaps} restrict the group homomorphism. Compatibility with hyperaddition follows by inspection of the definitions of hyperaddition given in Example~\ref{Examples:Hyperadditions}. For further details, we refer to \cite[Sec.~2.1]{maxwellGeometryTropicalExtensions2024}. 
    \end{proof}

We extend the hyperaddition on a hyperfied $\mathbb{H}$ to $m$ elements $x_1,\dots, x_m \in \mathbb{H}$ by (inductively) defining
\begin{align*}
    x_1 \oplus \dots \oplus x_m \coloneqq \bigcup_{x \in  x_1 \oplus \dots \oplus x_{m-1}} x \oplus x_m. 
\end{align*}

 \begin{lemma}
 \label{Lemma:TropExtensionSumMinimum}
      Let $\mathbb{H} \rtimes \Gamma$ be the tropical extension of an ordered abelian group $\Gamma$  by a hyperfield $\mathbb{H}$. For any $(x_1,w_1),\dots,(x_m,w_m) \in \mathbb{H} \rtimes \Gamma$, we have
       \begin{align}
        \label{Eq:HyperfieldSumMinimal}
       \bigoplus_{i=1}^m (x_i,w_i) =  \bigoplus_{i=1, \, w_i = W}^m (x_i,w_i) 
       \end{align}
      where $W \coloneqq \min_{i \in [m]} w_i$.
 \end{lemma}

\begin{proof}
    We prove the claim inductively. The case $m=2$ follows directly from Definition~\ref{Def:HyperAddition}. Assume that the claim is true for $m-1$. Thus, for $W_{m-1} \coloneqq \min_{i\in [m-1]} w_i$ we have 
          \begin{align}
        \label{Eq:HyperfieldSumMinimalProof}
          \bigoplus_{i=1}^{m-1} (x_i,w_i) =  \bigoplus_{i=1, \,w_i = W_{m-1}}^{m-1} (x_i,w_i) 
          \end{align}
     If $w_m = W_{m-1}$, then \eqref{Eq:HyperfieldSumMinimalProof} directly implies~\eqref{Eq:HyperfieldSumMinimal}.
      If $w_m < W_{m-1}$, then the hypersum of~\eqref{Eq:HyperfieldSumMinimalProof} and $(x_m,w_m)$ equals $(x_m,w_m)$, since for every element $(y,v)$ in~\eqref{Eq:HyperfieldSumMinimalProof} we have $v \geq W_{m-1} > w_m$. In particular, \eqref{Eq:HyperfieldSumMinimal} holds true.

      In the rest of the proof, we assume that $w_m > W_{m-1}$. Let $(y,v)$ be in the set \eqref{Eq:HyperfieldSumMinimalProof}. If $v < w_m$, then $(y,v) \oplus (x_m,w_m) = (y,v)$, which is contained in \eqref{Eq:HyperfieldSumMinimalProof}. If $v \geq w_m$, then $(y,v) \oplus (x_m,w_m)$ contains elements of the form $(y',v')$ where $v' \geq w_m > W_{m-1}$, but all elements with this property are contained in ~\eqref{Eq:HyperfieldSumMinimalProof} (cf. Definition~\ref{Def:HyperAddition}). This shows that the hypersum of \eqref{Eq:HyperfieldSumMinimalProof} and $(x_m,w_m)$ is \eqref{Eq:HyperfieldSumMinimalProof}, finishing the proof.
\end{proof}

\subsection{Model theory and o-minimal geometry}

Both of our main theorems (Theorems~\ref{Thm:A} and~\ref{Thm:B}) concern definable sets in a polynomially bounded o-minimal expansion of a real closed field. In this subsection, we recall the related necessary notions from model theory, following \cite{driesTameTopologyOminimal1998, markerModelTheory2002}.

\begin{definition}
    A \emph{structure} on a nonempty set $\cR$ is a sequence $S = (S_n)_{n \in \NN}$ such that, for each $n \in \NN$:
    \begin{itemize}
        \item $S_n$ is a boolean algebra of subsets of $\cR^n$, i.e. a collection of subsets of $\cR^n$ containing $\emptyset$ such that, if $A, B \in S_n$, then $A \cup B \in S_n$ and $\cR^n \setminus A \in S_n$;
        \item if $A \in S_n$, then $A \times \cR \in S_{n+1}$ and $ \cR\times A \in S_{n+1}$;
        \item $\{ \,(a_1, \dots , a_n) \in \cR^n \mid a_1 = a_n  \, \} \in S_n$; 
        \item if $\pi \colon \cR^{n+1} \to \cR^n$ is the projection onto the first $n$ coordinates and $A \in S_{n+1}$, then $\pi(A) \in S_n$.
    \end{itemize}
    Given a structure $S$ on $\cR$, we say that $A \subset \cR^n$ is \emph{definable} if $A \in S_n$. We say that a function is definable if its graph is definable. 
\end{definition}
We are interested in sets equipped with an ordering, such as real closed fields and, more generally, ordered groups. A (linearly) ordered set $(\cR, <)$ is called \emph{dense} if for all $a, b \in \cR$ with $a<b$, there exists $c \in \cR$ such that $a<c<b$. We say that $\cR$ has \emph{no endpoints} if $\cR$ has no smallest or largest element.
Following the French tradition, we denote by $]a, b[$ the open intervals in $\cR$, and by $[a, b]$ the closed intervals. For all $n \in \NN$, we equip $\cR^n$ with the box topology, i.e. the topology generated by the open sets of the form $]a_1, b_1[\,\times \dots \times \, ]a_n , b_n[$. In case $\cR$ is a real closed field, we also call the box topology the \emph{strong topology}, as usual.
\begin{definition}
    Let $(\cR, <)$ be a dense linearly ordered nonempty set without endpoints. We call a structure $S$ on $(\cR, <)$ \emph{order-minimal} (or \emph{o-minimal}) if 
    \begin{itemize}
        \item $\{ \, (a,b) \in \cR^2 \mid a<b \,\} \in S_2$;
        \item the sets in $S_1$ are exactly finite unions of (possibly unbounded) intervals and points.
    \end{itemize}
\end{definition}

A structure $S$ on a set $\cR$ is usually introduced specifying relations and functions that "generate" $S$. For instance, the structure over a real closed field $\KK$ generated by addition, multiplication, and by $\KK$ (identified with the set of functions from $\KK^0 = \{0\}$ to $\KK$) will consist exactly of semialgebraic sets.

\begin{definition}
    A \emph{model-theoretic structure} $\fR = (\cR, (R_i)_{i \in I}, (f_j)_{j\in J})$ consists of a nonempty set $\cR$, of relations $R_i \subset \cR^{m(i)}$ for $i \in I$ and $m(i) \in \NN$, and of functions $f_j \colon \cR^{n(j)} \to \cR$ for $j \in J$ and $n(j) \in \NN$. Given a model-theoretic structure $\fR$ on $\cR$, we denote $\Def(\fR)$ to be the smallest structure on $\cR$ that contains each relation and the graph of each function in $\fR$. Moreover, for a subset $C \subset \cR$, we denote $\fR_C$ the model-theoretic structure obtained by adding to the functions of $\fR$ the constant functions $c \colon \cR^0 \to \cR$ for $c \in C$. A set $A$ is \emph{definable in $\fR$ with parameters from $C$}, or simply \emph{$C$-definable}, if it is definable in $\fR_C$.
\end{definition}

\begin{definition}
\begin{itemize}
\item[]
\item[(a)] Let $(\cR , <)$ be a dense linearly ordered group without endpoints. A model-theoretic structure $\fR = (\cR, (R_i)_{i \in I}, (f_j)_{j\in J})$ is called \emph{o-minimal} if 
$<$ gives a relation $R_{i_*}$ for some $i_* \in I$, and $\Def(\fR_\cR)$ is an o-minimal structure. 
\item[(b)]
An o-minimal structure  $\fR = (\cR, (R_i)_{i \in I}, (f_j)_{j\in J})$ \emph{expands} an ordered abelian group $\cR$ (respectively an ordered abelian ring $\cR$) if there are functions $0 \colon \cR^0 \to \cR$, $+ \colon \cR^2 \to \cR$, $- \colon \cR \to \cR$ (and $\cdot \colon \cR^2 \to \cR$ for an ordered abelian ring) that are definable and such that $(\cR, 0, +, -)$ is an abelian group (respectively such that $(\cR, 0, +, -, \cdot)$ is an ordered abelian ring).
\end{itemize}
\end{definition}

Definable sets in o-minimal expansions of a real closed fields enjoy most of the properties that usual semialgebriac sets have. We refer the reader to \cite{driesTameTopologyOminimal1998, costeIntroductionOminimalGeometry1999} for detailed accounts of such properties.

Next, we recall an equivalent characterization of definable sets.
A \emph{language} $\cL$ is the union of two sets: a set of \emph{relation symbols} and a set of \emph{function symbols}. Each relation symbol $R$ (respectively each function symbol $f$) is equipped with a number $n_R \in \NN$ (respectively $m_f \in \NN$). We say that that $R$ is $m_R$-ary (respectively that $f$ is $n_f$-ary). $0$-ary function symbols are called \emph{constant symbols}. We denote $\cL_{C}$ the extension of $\cL$ with the constant symbols $c \in C$.

If $\cL$ is a language with relation symbols $(R_i)_{i\in I}$ and function symbols $(f_j)_{j \in J}$, then an \emph{$\cL$-structure} is a model-theoretic structure $\fR=(\cR, (R_i^\cR)_{i\in I}, (f_j^\cR)_{j \in J})$, where $\cR$ is a nonempty set, such that, for every $m$-ary relation symbol $R_i$ and every $n$-ary function symbol $f_j$, we have a corresponding definable set $R_i^\cR \subset \cR^m$ and  a definable function $f_j^\cR \colon \cR^n \to \cR$, called the \emph{interpretations} of $R$ and $f$ in $\fR$. By abuse of notation, we will drop the superscript $\cdot^\cR$, and use the same symbol for the symbols and their interpretations. 

An \emph{$\cL$-term} is an expression constructed using variables and function symbols from $\cL$. An \emph{atomic $\cL$-formula} is an expression containing equality of terms or relations (from $\cL$) of terms. The set of $\cL$-formulas is the closure of the set of atomic $\cL$-formulas under negation $\neg$, and $\wedge$, or $\vee$, and existential quantifiers $\forall$, $\exists$.

We will write $\phi(x_1, \dots , x_n)$ to make explicit the free variables $x_1, \dots , x_n$ of an $\cL$-formula $\phi$. If $\phi(x_1, \dots , x_n)$ is a $\cL$-formula and $\fR$ is an $\cL$-structure, and $a\in \cR^n$, we write 
$$\fR \models \phi(a)$$ 
if $\phi(a)$ is true in $\fR$ (we refer to \cite[Def.~1.1.6]{markerModelTheory2002} for a precise definition). Using this terminology, definable sets in an $\cL$-structure can be described as follows.
\begin{proposition}(\cite[Prop.~1.3.4]{markerModelTheory2002})
    \label{rem:definable}
    Let $\cL$ be a language and $\fR$ be an $\cL$-structure. A set $X \subset \cR^n$ is definable in $\fR$ if and only if
\begin{align*}
    X &= \{ \, a \in \cR^n \mid \fR \models \phi(a,b) \, \}, \qquad \text{ for an $\cL$-formula } \phi(x_1, \dots , x_n, y_1, \dots , y_m)  \text{ and } b \in \cR^m, \text{ or } \\
    X &= \{ \, a \in \cR^n \mid \fR \models \psi(a) \, \} \qquad \quad \text{ for an $\cL_\cR$-formula } \psi(x_1, \dots , x_n). 
\end{align*}
\end{proposition}

\begin{remark}
    Given a model-theoretic structure $\fR = (\cR, (R_i)_{i \in I}, (f_j)_{j\in J})$ one can always construct a language $\cL$ such that $\fR$ becomes an $\cL$-structure by simply introducing a relation symbol for each relation $R_i$ and a function symbol for each function $f_j$.
\end{remark}

If $\fR$ and $\mathfrak{S}$ are $\cL$-models and $\fR \subset \mathfrak{S}$ (defined in the obvious way), we say that $\mathfrak{S}$ is an \emph{elementary extension} of $\fR$ if, for all $\cL$-formulas $\phi(x_1, \dots , x_n)$ and $a \in \cR^n$, we have
    $
        \fR \models \phi(a) \Longleftrightarrow \mathfrak{S} \models \phi(a) 
    $.
For a definable set $X\subset \cR^n$, we denote its $\mathfrak{S}$-extension by
\[  X_\mathfrak{S} \coloneqq \{\, b \in \mathcal{S}^n \mid \mathfrak{S} \models \psi(b)\,\},\]
where $\psi(x_1,\dots,x_n)$ is any $\mathcal{L}_\mathcal{R}$ formula defining $X$ as in ~\Cref{rem:definable}. We write $\fR \prec \mathfrak{S}$ if the extension is elementary.

We conclude this subsection by discussing a crucial assumption on o-minimal structures needed to prove Theorem~\ref{Thm:A} and~\ref{Thm:B}, namely polynomial boundedness.
\begin{definition}
    Let $\fF$ be an o-minimal expansion of a real closed field $(\FF, <, + , \, \cdot \,)$. A \emph{power function} of $\fF$ is a definable function $f \colon \FF_{>0} \to \FF_{>0}$, not identically zero, such that $f(rs) = f(r)f(s)$ for all $r,s \in \FF_{>0}$. The set of power functions is a field with the operations of multiplication and composition. 
    
The derivation map gives an embedding $f \mapsto f'(1)$ of the field of power functions into $\FF$. The image of this map is called \emph{field of exponents}, which we denote by $\Lambda= \Lambda(\fF)$. 
For $\lambda = f'(1)$, we write $f = x^\lambda$ and $f(s) = s^\lambda$. The usual properties of monomial powers still hold true: $(rs)^\lambda = r^\lambda s^\lambda$, $(s^{\lambda_1})^{\lambda_2} = s^{\lambda_1 \lambda_2}$, $s^{\lambda_1}  s^{\lambda_2} = s^{\lambda_1 
    +\lambda_2}$,. We say that $\fF$ is \emph{power bounded} if for every definable function $f \colon \FF \to \FF$ there exists $\lambda \in \Lambda$ such that $\abs{f(s)} \le s^\lambda$ for all $s \in \FF$ sufficiently large. We say that $\fF$ is \emph{polynomially bounded} if it is power bounded and the field of exponents is archimedean, i.e., $\Lambda \subset \RR$.
\end{definition}

\begin{remark}
    \label{rem:exponets_inside_value_group}
    Let $\fF$ be a polynomially bounded o-minimal expansion of a real closed field $\FF$, and equip $\FF$  with an order compatible nontrivial valuation.
    After fixing a splitting and normalizing the valuation as in Assumption~\ref{assumtion:valued_fields}, it follows that
    \[\val(t^\lambda) = \lambda \val(t) = \lambda.\]
     In particular, the field of exponents $\Lambda(\fF)$ is contained in the value group $\Gamma_\FF$.
\end{remark}

For a list of polynomially bounded o-minimal structures, we refer the reader to \cite[p.~505]{driesGeometricCategoriesOminimal1996b}. Semialgebraic sets are the most elementary example. For an equivalent characterizations of polynomially bounded o-minimal structures, see \cite{millerExpansionsRealField1994,millerGrowthDichotomyOminimal1996}. We have the following explicit description of definable functions in polynomially bounded o-minimal structures.

\begin{theorem}[{Preparation theorem for polynomially bounded o-minimal structures \cite[Th.~2.1]{driesOminimalPreparationTheorems2002}, \cite[p.~85]{nowakProofValuationProperty2007}}]
    \label{thm:preparation_theorem}
    Let $\fF$ be a polynomially bounded o-minimal structure expanding a real closed field $\FF$, $f \colon \FF^{n+1} \to \FF$ be a definable function, and $\epsilon \in \QQ_{>0}$. Then there exists definable sets $C_1, \dots, C_m \subset \FF^{n+1}$, exponents $\lambda_1, \dots , \lambda_m \in \Lambda$, and definable functions $r_1, \dots , r_m, c_1, \dots , c_m \colon \FF^n \to \FF$, $u \colon \FF^{n+1} \to ]1-\epsilon, 1+\epsilon[$ s.t., for all $(s, a) \in C_i$:
    \[
        f(s,a) = \abs{s-r_i(a)}^{\lambda_i} c_i(a) u(s,a).
    \]
\end{theorem}

\subsection{Model theory of valued real closed fields}
\label{subsec:van_den_Dries}
The study of o-minimal expansions of valued fields from a model-theoretic viewpoint started in \cite{driesTConvexityTameExtensions1995, driesTConvexityTameExtensions1997}. We recall in this section their fundamental results.
In the following, $\FF$ denotes a real closed field equipped with an order compatible valuation $\val$. Since we will mostly apply these results in the case where the valuation is nontrivial, we denote the field by $\FF$. Nevertheless, the results discussed in this section are valid even when the valuation on $\FF$ is trivial.

If $\fF$ is a polynomially bounded o-minimal expansion of a real closed field $\FF$, every convex valuation ring $\cO \subset \FF$ is \emph{T-convex} in the sense of \cite{driesTConvexityTameExtensions1995, driesTConvexityTameExtensions1997}, see \cite[(4.2)~Prop.]{driesTConvexityTameExtensions1995}. Moreover, if $\cL$ is the language of $\fF$, and if we fix an order compatible valuation $\val_{\FF}$ (or equivalently, a convex valuation ring $\cO_\FF$), then the residue field $k_\FF$ inherits an $\cL$-structure $\mathfrak{k}_\fF$ from $\fF$, see \cite[(2.6)~Rem.]{driesTConvexityTameExtensions1995}. The structure $\mathfrak{k}_\fF$ arises by identifying $k_{\FF}$ with any subfield of $\cO_\FF$ maximal with respect to inclusion.

We can also induce a structure $\Gamma_{\FF, \mathrm{vg}}$ on the value group $\Gamma_\FF$, by introducing a relation symbol for the image under the valuation map of every definable (without paramaters) function and relation in $\fF$, see \cite[(3.15)]{driesTConvexityTameExtensions1995}. The structure $\Gamma_{\FF, \mathrm{vg}}$ is o-minimal in the polynomially bounded case \cite[(4.5)~Cor.]{driesTConvexityTameExtensions1995}. The value group $\Gamma_{\FF}$ is an ordered vector space over the field of exponents $\Lambda = \Lambda(\fF)$ (which we identify as a subspace of $\Gamma_\FF$ as in \Cref{rem:exponets_inside_value_group}). It is proven in \cite[Th.~B]{driesTConvexityTameExtensions1997} that the definable sets in $\Gamma_{\FF, \mathrm{vg}}$ coincide exactly with the definable subsets of $\Gamma_\FF$, seen as an ordered vector space over $\Lambda$. Such sets are called \emph{semilinear}, and we refer the reader to \cite[p.~25--28]{driesTameTopologyOminimal1998} for their precise description.

Finally, we note that some of the results of \cite{driesTConvexityTameExtensions1995,driesTConvexityTameExtensions1997} have been recently rediscovered in \cite[Sec.~3~and~4]{Allamigeon2020} in the more restrictive semialgebraic setting. 
\section{Tropicalization and initial degenerations}
\subsection{Tropicalizing definable sets}
\label{Sec:TropDefSets}
In this section, we begin investigating the tropicalization of definable sets. From now on we will work under the following assumption. 

\begin{assumption}
    \label{ass:elementary_extensions}
    Let $\fK \prec \fF \prec \fM$ be elementary extensions of polynomially bounded o-minimal expansions of the real closed fields $\KK \subset \FF \subset \MM$. Furthermore, assume that $\KK \subset \FF \subset \MM$ is an extension of valued real closed fields as in Assumption~\ref{assumtion:valued_fields}.
\end{assumption}

The \emph{tropicalization} of a definable set $X \subset \KK_{>0}^n$ is defined as
\begin{align}
\label{Def:Trop}
\Trop(X) \coloneqq \overline{ \val(X_{\fF})}
\end{align}
where the closure is taken in the euclidean topology of $\RR^n$. In the following, we investigate which points are added to $\val(X_\fF)$ when taking its closure.

\begin{proposition}
    \label{prop:ValImageClosedM}
    Let $\fF$ be a polynomially bounded o-minimal expansion of a real closed field $\FF$, equipped with nontrivial order compatible valuation $\val \colon \FF^\times \to \Gamma_\FF$. Then, for $X \subset \FF^n_{>0}$ definable, $\val(X)$ is a semilinear closed subset of $\Gamma^n_\FF$, and $\val(X) = \val(\cl{X})$ (where the closure is the relative closure of $X$ in $\FF_{>0}^n$).
\end{proposition}

\begin{proof}
    From \Cref{subsec:van_den_Dries}, $\val(X)$ is a semilinear subset  $\Gamma_\FF^n$.    
    By \cite[(4.5)~Cor.]{driesTConvexityTameExtensions1995}, the structure on $\Gamma_\FF$ induced by the valuation is o-minimal. Take $w$ to be in the closure of ${\val(X)}$ in $\Gamma_\FF^n$. By the o-minimal curve selection lemma, see e.g. \cite[Th.~3.2]{costeIntroductionOminimalGeometry1999}, there exists a continuous definable curve \[
        \phi \colon \Gamma_\FF \supset [0, v] \longrightarrow \val(X) \subset \Gamma^n_\FF
    \]
    such that $\phi(]0,v]) \subset X$ and $\phi(0) = w$. Up to extending the language of $(\fF, \cO_\FF)$ as in \cite[p.~21]{driesTConvexityTameExtensions1997}, we have definable Skolem functions \cite[Rem.~2.7]{driesTConvexityTameExtensions1997}, i.e. definable sections of maps always exist (see \cite[Ex.~2.5.25]{markerModelTheory2002} for the precise definition). Therefore there exists a continuous definable function \[f \colon  \FF \supset [0, t^v] \longrightarrow \cl{X} \] (where the closure is taken inside $\FF^n_{>0}$) such that $\val(f(a)) = \phi(\val(a))$ and $\phi(]0,t^v]) \subset X$. Definable sets in the value group are precisely $\Gamma$-affine $\Lambda$-integral functions; hence there exists $\lambda \in \Lambda^n$ and $b \in \Gamma_\FF^n$ such that, for all $\gamma \in [0, v]$ small enough, $\phi(\gamma) = \gamma \lambda + b$. But then, for all $a$ small enough:
    \[
        \val(a) \lambda + b = \phi(\val(a)) = \val(f(a)) \xrightarrow{a \to 0} \val(f(0)) \in \val(\cl{X}) \subset \Gamma_\FF^n
    \]
    Since $\val(a) \to \infty$ as $a \to 0$, continuity of $\phi$ forces $\lambda =0$.
    This means that the valuation $\val(f(a))$ is finally constant and equal to $w = \val(b)$. Hence there exists $a \in X$ such that $\val(a) = w$, proving that ${\val(X)}$ is closed.
    
    For the last statement, recall that $\cl{X}$ is definable and that $\val \colon \FF^n_{>0} \to \Gamma^n_\FF$ is continuous, and then $\val(X) = \overline{\val(X)} = \overline{\val(\cl{X})} = \val(\cl{X})$.
\end{proof}

We will need the following weak version of the strong density theorem, (c.f. \Cref{thm:strong_density}), which follows from quantifier elimination. \begin{proposition}
    \label{prop:density}
Let $\fF \prec \fM$ and $\FF \subset \MM$ be as in Assumption~\ref{ass:elementary_extensions}. Then, for $X \subset \FF^n$ definable:
    \begin{align*}
        \val(X_\fM) \cap \Gamma_{\FF}^n = \val(X) 
      \end{align*}
\end{proposition}
\begin{proof}
    If a structure is polynomially bounded, then its convex valuation ring is $T$-convex in the sense of \cite{driesTConvexityTameExtensions1995,driesTConvexityTameExtensions1997}. Up to extending the language of valued real closed fields as in \cite[Rem.~3.11]{driesTConvexityTameExtensions1995}, the quantifier elimination theorem holds true \cite[Rem.~3.11]{driesTConvexityTameExtensions1995}, and then the set $\val(X_\fF) \subset \Gamma_\FF^n$ is defined by a quantifier-free formula. By model completeness, which follows from quantifier elimination (see e.g. \cite[Sec.~3.3 and Th.~3.4.3]{prestelMathematicalLogicModel2011}), quantifier free formulas which are true over $\fM$ are true over $\fF$. Hence, if $w \in \val(X_\fM) \cap \Gamma_\FF^n$ then there exists $a \in X$ such that $\val_\FF(a_i) = w_i$ for $i \in  [n]$.
\end{proof}

We finish this section by showing that no additional $\Gamma_\FF$-valued points are added to the tropicalization in~\eqref{Def:Trop} when taking the closure.

\begin{corollary}
       \label{cor:ValImageClosure}
      Let $\fF$ be a polynomially bounded o-minimal expansion of a real closed field $\FF$, equipped with a nontrivial order compatible valuation $\val \colon \FF^\times \to \Gamma_{\FF}$. For each definable set $X \subset \FF^n_{>0}$, we have
    \begin{align*}
        \val(X)  =  \overline{\val(X)}  \cap \Gamma_\mathbb{F}^n
      \end{align*}
      where the closure is taken in $\RR^n$.
\end{corollary}

\begin{proof}
    Let $\fF \prec \fM$ be an elementary extension equipped with a valuation extending the valuation of $\FF$ such that $\val(\MM^\times) = \RR$. Such an extension exists by \cite[(4.4)~Prop.]{driesTConvexityTameExtensions1995}. 
    The claim follows by considering the following chain of inclusions
    \begin{align*}
        \val(X) \subset \overline{\val(X)} \cap \Gamma_\FF^n  \overset{(1)}{=} \overline{ \val(X_\fM) \cap \Gamma_\FF^n } \cap \Gamma_\FF^n \subset \overline{ \val(X_\fM)} \cap \overline{\Gamma_\FF^n } \cap \Gamma_\FF^n \overset{(2)}{\subset}  \val(X_\fM) \cap  \Gamma_\FF^n \overset{(3)}{=} \val(X),
    \end{align*}
    where $(1)$ and $(3)$ follows from Proposition~\ref{prop:density} and $(2)$ holds by Proposition~\ref{prop:ValImageClosedM}.
\end{proof}

\subsection{Tropicalizations and the initial form of a polynomial}
\label{Section:TropofDefFunctions}
In this section, we assume that $\FF$ is a nontrivially valued real closed field and $\Lambda \subset \RR^n$ is a field of exponents.
The \emph{real tropicalization} and the \emph{tropicalization} of a polynomial $f = \sum_\lambda c_\lambda \vb x^\lambda \in \FF[\vb x^\Lambda]$ are respectively defined as 
\begin{align*}
   \trop_r(f) &\coloneqq \bigoplus_\lambda \left( \sign(c_\lambda), \val(c_\lambda) \right) \odot \overline{\vb w}^\lambda \in \mathbb{RT}_\FF[\overline{\vb w}^\Lambda] \\
   \trop(f) & \coloneqq \bigoplus_\lambda \val(c_\lambda) \odot \vb w^\lambda \in \mathbb{T}_\FF[\vb w^\Lambda]
\end{align*}
where the addition and multiplications are interpreted in the hyperfields $\mathbb{RT}_\FF$ and $\mathbb{T}_\FF$ respectively, cf. Section~\ref{Sec:Hyperfields}.
Following \cite[Section 5.2]{jellRealTropicalizationAnalytification2020}, we write
\begin{equation}
\label{Eq:OrderRealTrop}
\begin{aligned}
 \trop_r(f)(\overline{w}) < 0 :\Longleftrightarrow \trop_r(f)(\overline{w}) \cap ( \{1\} \times \Gamma \cup \{\infty\} ) = \emptyset, \\
    \trop_r(f)(\overline{w}) \geq 0 :\Longleftrightarrow \trop_r(f)(\overline{w}) \cap ( \{1\} \times \Gamma \cup \{\infty\} ) \neq \emptyset.
\end{aligned}
\end{equation}
These conventions are chosen precisely so that  $ \trop_r(f)(\overline{w}) \geq 0 $ and $ \trop_r(f)(\overline{w}) < 0 $  are logical negations of each another.
 In light of these definitions, it is natural to introduce the \emph{RV-tropicalization} of $f$ as
\begin{align}
\label{Eq:RVPolynomials}
 \trop_{\rv}(f)(z) \coloneqq  \bigoplus_\lambda \rv(c_\lambda) \odot  z^\lambda 
 =  \bigoplus_\lambda (\ac(c_\lambda),\val(c_\lambda)) \odot  z^\lambda \in \RV_\FF[\vb z^\Lambda]
 \end{align}
where now we interpret the addition and the multiplication in the $\RV$-tropical hyperfield. Similarly to \eqref{Eq:OrderRealTrop}, using the ordering on $\RV_\FF$ discussed in Section~\ref{Sec:RealClosedValuedFields}, for $z \in \RV_\FF$ we have 
\begin{align*}
   \trop_{\rv}(f)(z) < 0 \Longleftrightarrow \trop_{\rv}(f)(z) \cap ( k_{\FF,>0}  \times \Gamma_\FF \cup \{0\} ) = \emptyset, \\
    \trop_{\rv}(f)(z) \geq 0 \Longleftrightarrow \trop_{\rv}(f)(z) \cap ( k_{\FF,>0} \times \Gamma_\FF \cup \{0\} ) \neq \emptyset.
\end{align*}

\begin{lemma}
\label{Prop:PolyRVTrop}
    Consider the hyperfield morphism $\rv\colon \mathbb{F} \to \RV_\FF$
    from Lemma~\ref{Lemma:HyperfieldMorphisms} and extend this  coordinate wise to $\rv\colon \mathbb{F}_{>0}^n \to \RV_{\FF,>0}^n$.
   For a polynomial $f = \sum_\lambda c_\lambda \vb x^\lambda \in \FF[\vb x^\Lambda]$ and for each $z \in \RV_{\FF,>0}^n$, we have
    \begin{align*}
         \rv\left( f\left( \rv^{-1}(z) \right) \right) {\subset} \trop_{\rv}(f)(z).
    \end{align*}
\end{lemma}

\begin{proof}
Let $u \in \rv(f(\rv^{-1}(z)))$. Hence there exists $a \in \rv^{-1}(z)$ such that
\[
    u = \rv (f(a)) = \rv(\sum_{\lambda} c_\lambda a^\lambda) \subset \bigoplus_\lambda \rv(c_\lambda a^\lambda) = \bigoplus_\lambda \rv(c_\lambda) \odot \rv(a)^\lambda = \trop_{\rv}(z)
\]
where the inclusion follows since $\rv$ is a hyperfield morphism.
\end{proof}

One might be tempted to believe that the converse inclusion holds in Lemma~\ref{Prop:PolyRVTrop}. The following example shows that this is not the case.
\begin{example}
\label{Ex:SmoothConditions}
    Consider the polynomial $f = (x-1)^2 = 1 - 2x +x^2 \in \FF[x]$. Since $f(a) \geq 0$ for all $a\in\FF_{>0}^n$, we have $\rv(f(\rv^{-1}(z))) \subset \RV_{\FF,>0} \cup \{0\}$ for each $z \in \RV_{\FF,>0}$. On the other hand, consider the RV-tropicalization
    \[\trop_{\rv}(f)(z) = (1,0) \odot z^0  \oplus (-2,0) \odot z \oplus (1,0)\odot z^2.\]
    Evaluating this polynomial at $z = (1,0)$ we have 
 \[\trop_{\rv}(f)(1,0) = (1,0) \oplus (-2,0) \oplus (1,0) =  \left\{\,(y, v) \mid v\in \Gamma_\FF,\, v>0,\, y \in k^{\times}_\FF\,\right\} \cup\{0\},\]
 which strictly contains the set $\RV_{\FF,>0} \cup \{0\}=(k_{\FF,>0} \times \Gamma_\FF) \cup \{0\}$.
\end{example}

\begin{definition}
The \emph{initial form} of a polynomial $f = \sum_{\lambda}c_{\lambda}x^\lambda \in \mathbb{F}[\vb x^\Lambda]$ is defined as  
\begin{align*}
    \initial_w(f)(y) \coloneqq \sum_{\substack{\lambda \in \Lambda^n\\ \val(c_\lambda) + w\cdot \lambda = W}} \res_{\mathbb{F}}\big( c_\lambda t^{-\val(c_\lambda)}\big)\;y^\lambda \in k_\FF[\vb y^\Lambda], \quad  \text{ where } W =  \min\{\, \val(c_\lambda) + w \cdot \lambda \mid \lambda \in \Lambda^n, \, c_\lambda \neq 0\,\}.\end{align*}
\end{definition}

To study initial forms and initial degenerations, it is convenient to consider the twist of an element $b \in \FF^n_{>0}$
 by a vector $w \in \Gamma_\FF^n$. To that end, we introduce the following notation.
\[t^w.b \coloneqq (t^{w_1}b_1,\dots,t^{w_n}b_n ) \in \FF^n_{>0}.\]
For a set $X \subset \FF^n_{>0}$, we also write $t^w.X \coloneqq \{ \, t^w.b \mid b \in X \, \}$.
In the next lemma, we give a description of initial forms using the angular component map. Switching between these two descriptions will be useful for relating initial forms to $\RV$-tropicalizations.

\begin{lemma}
\label{prop:InitilFormPolynomials}
    Let $f = \sum_{\lambda}c_{\lambda}x^\lambda \in \mathbb{F}[\vb x^\Lambda]$ be a polynomial and $W \coloneqq \min\{\, \val(c_\lambda) + w \cdot \lambda \mid \lambda \in \Lambda^n, \, c_\lambda \neq 0\,\}$. For all $y \in k_{\FF,>0}^n$ and $b \in \mathcal{O}^n_{\FF}$ with $\res_\FF(b) = y$, we have
\begin{align*}
    \initial_w(f)(y) \;=\; \res_{\mathbb{F}}\big(t^{-W} f(t^w.b )\big) \; =
    \sum_{\substack{\lambda \in \Lambda^n\\ \val(c_\lambda) + w\cdot \lambda = W}} \ac_\FF ( c_\lambda)\;y^\lambda.
\end{align*}
\end{lemma}
\begin{proof}
The first equality is well-known in the tropical geometry literature, and it follows directly from
\[   \initial_w(f)(y) =    \sum_{\substack{\lambda \in \Lambda^n\\ \val(c_\lambda) + w\cdot \lambda = W}}  \res_{\mathbb{F}}\big( c_\lambda t^{w\cdot \lambda - W} \big) \res_{\mathbb{F}}( \;b)^\lambda  = \sum_{\lambda \in \Lambda^n}  \res_{\mathbb{F}}\big( c_\lambda t^{w\cdot \lambda - W} \big) \res_{\mathbb{F}}( \;b^\lambda ) = \res_{\mathbb{F}}\big(t^{-W} f(t^w.b )\big)\]
since $\res_{\mathbb{F}}\big( c_\lambda t^{w\cdot \lambda - W} \big) =0$ for all $\lambda \in \Lambda^n$ with $w\cdot \lambda - W >0$.

The second equality follows by observing that $\ac_\FF(c_\lambda) = \res_\FF(t^{-\val(c_\lambda)}c_\lambda)$.
\end{proof}

The initial form and the $\RV$-tropicalization of a polynomial are related as follows.

\begin{lemma}
\label{Lemma:RVvsInitialForms}
Let $f = \sum_{\lambda \in \Lambda^n} c_\lambda x^\lambda \in \FF[\vb x^\Lambda]$ be a polynomial and $W \coloneqq \min\{\, \val(c_\lambda) + w \cdot \lambda \mid \lambda \in \Lambda^n, \, c_\lambda \neq 0\,\}$. For each $(y,w) \in \RV_{\FF,>0}^n$, we have
\begin{align*}
    \trop_{\rv}(f)(y,w) = \begin{cases}
        \{\,(\initial_w(f)(y),W) \,\} \quad &\text{ if } \initial_w(f)(y) \neq 0 \\
        \{(s,v) \mid v \in \Gamma_\FF, v > W, s \in k^\times_\FF \} \cup \{0\} \quad &\text{ if }  \initial_w(f)(y) = 0.
    \end{cases}
\end{align*}
\end{lemma}

\begin{proof}
Using the definition of the hyperfield operations in $\RV_\FF$ from Example~\ref{Examples:Hyperadditions}, we rewrite $\trop_{\rv}(f)(y,w)$:
    \begin{align*}
    \trop_{\rv}(f)(y,w) &=  \bigoplus_{\lambda \in \Lambda^n} (\ac_\FF(c_\lambda),\val(c_\lambda)) \odot (y,w)^\lambda = \bigoplus_{\lambda\in \Lambda^n}  (\ac_\FF(c_\lambda) y^\lambda, \val(c_\lambda) + \lambda \cdot w) \\
    &\overset{(*)}{=} \bigoplus_{\substack{\lambda \in \Lambda^n\\ \val(c_\lambda) + w\cdot \lambda = W}} (\ac_\FF(c_\lambda) y^\lambda, W) \\ &= \begin{cases}
       \{\, ( \sum_{\val(c_\lambda) + w\cdot \lambda = W} \ac_\FF ( c_\lambda)\;y^\lambda, W) \, \} \quad & \text{ if } \sum_{\val(c_\lambda) + w\cdot \lambda = W} \ac_\FF ( c_\lambda)\;y^\lambda \neq 0 \\
        \{(s,v) \mid v \in \Gamma_\FF, v > W, s \in k^\times_\FF \} \cup \{0\}, \quad & \text{ if } \sum_{\val(c_\lambda) + w\cdot \lambda = W} \ac_\FF ( c_\lambda)\;y^\lambda = 0
    \end{cases}
\end{align*}
The equality $(*)$ follows from Lemma~\ref{Lemma:TropExtensionSumMinimum}. Using Lemma~\ref{prop:InitilFormPolynomials}, we conclude the proof by observing that  $\initial_w(f)(y) = \sum_{\val(c_\lambda) + w\cdot \lambda = W} \ac_\FF ( c_\lambda)\;y^\lambda $.
\end{proof}

We conclude this section with a technical result that will be used in the proofs of \Cref{Thm:WeakFundamentalFineTrop,Thm:WeakFundamentalTrop}.

\begin{proposition}
\label{Eq:TropPrecludingPosValues}
\label{Lemma:TropPrecludingPosValues}
Let $f = \sum_{\lambda \in \Lambda^n} c_\lambda x^\lambda \in \FF[\vb x^\Lambda]$ be a polynomial and let $w \in \Gamma_\FF^n$.
\begin{itemize}
    \item[(a)]  For all $(y,w) \in \RV_{\FF,>0}^n$, we have $\trop_{\rv}(f)(y,w) \geq  0$ if and only if $\initial_w(f)(y) \geq 0$.
\item[(b)]  If  $\trop_{\rv}(f)(y,w) \geq 0$ for some $y \in k_{\FF,>0}^n$, then $\trop_r(f)(1,w) \geq 0$.
    \item[(c)]  If $\trop_r(f)(1,w) < 0$, then $\trop_{\rv}(f)(y,w) < 0$ for all $y \in k_{\FF,>0}^n$.
\end{itemize}
\end{proposition}
\begin{proof}
We prove (a) by applying repeatedly \Cref{Lemma:RVvsInitialForms}. Let $W = 
\min\{\, \val(c_\lambda) + w \cdot \lambda \mid \lambda \in \Lambda^n, \, c_\lambda \neq 0\,\}$.
If $\initial_w(f)(y) > 0$, then $ \trop_{\rv}(f)(y,w) = \{\,(\initial_w(f)(y),W) \,\} $ by \Cref{Lemma:RVvsInitialForms}, and $\trop_{\rv}(f)(y,w) > 0$.
If $\initial_w(f)(y) =0$, using again \Cref{Lemma:RVvsInitialForms} we have that  $\trop_{\rv}(f)(y,w) \cap (k_{\FF,>0} \times \Gamma_\FF \cup \{0\}) \neq \emptyset$, which by definition is equivalent to $\trop_{\rv}(f)(y,w) \geq 0$.
Now assume that $\trop_{\rv}(f)(y,w) \geq 0$. If $0 \in \trop_{\rv}(f)(y,w)$, then \Cref{Lemma:RVvsInitialForms} implies that $\initial_w(f)(y) = 0$.
If $0 \notin \trop_{\rv}(f)(y,w)$, then again by \Cref{Lemma:RVvsInitialForms} we have $\initial_w(f)(y) > 0$.

To prove part (b), consider the hyperfield morphism $\phi \colon \RV_\FF \to \mathbb{RT}_\FF$ from \Cref{Lemma:HyperfieldMorphisms}. Let $y \in k_{\FF,>0}^n$ such that $\trop_{\rv}(f)(y,w) \geq 0$, that is,  $\trop_{\rv}(f)(y,w)  \cap (k_{\FF,>0} \times \Gamma_\FF \cup \{0\})) \neq \emptyset$. Therefore 
\[ \emptyset \neq \phi\big( \trop_{\rv}(f)(y,w)  \cap (k_{\FF,>0} \times \Gamma_\FF \cup \{0\})  \big) \subset  \phi\big( \trop_{\rv}(f)(y,w)\big)  \cap \big(\{1\} \times \Gamma_\FF \cup \{0\}\big).\]
Since $\phi$ is a hyperfield morphism and $\phi(y,w) = (1,w)$, we also have
\begin{align*}
    \phi(   \trop_{\rv}(f)(y,w)  ) = \phi \left(  \bigoplus_\lambda (\ac_\FF(c_\lambda),\val(c_\lambda)) \odot  (y,w)^\lambda  \right) \subset \bigoplus_\lambda (\sign(c_\lambda),\val(c_\lambda)) \odot  (1,w)^\lambda = \trop_r(f)(1,w)
\end{align*}
which implies that the intersection $\trop_r(f)(1,w) \cap (\{1\} \times \Gamma \cup \{0\})$ is nonempty, showing that $\trop_r(f)(1,w) \geq 0$. Part (c) is simply the negation of (b). 
\end{proof}

\subsection{Initial degenerations of definable sets}
\label{Sec:InitialDegen}
Initial degenerations of algebraic varieties have been studied extensively in tropical geometry \cite{SpeyerSturmfels,Speyer_Thesis,Payne,Katz, Gubler}. 
In this section, we extend the study of initial degenerations to definable sets.
Let $\FF$ be a nontrivially valued real closed field.
Inspired by \cite[Sec.~4.3]{Alessandrini},  we call 
    \begin{align}
    \label{Def:InitialDegeneration}
         \initial_w(X) \coloneqq \res_\FF(t^{-w}.X \cap (\cO_{\FF,>0}^\times)^n)  \subset k_{\FF, >0}^n 
    \end{align}
  the \emph{initial degeneration} of a set $X \subset \FF^n_{>0}$ with respect to $w \in \Gamma_{\FF}^n$.
  
\begin{proposition}
    \label{lem:initial_closed_dimension}
    Let $\fF$ be a polynomially bounded o-minimal expansion of a real closed field $\FF$, equipped with a nontrivial order compatible valuation. Then, for any definable set $X \subset \FF_{>0}^n$ and $w \in\Gamma_\FF$, the set $\initial_w(X)$ is a closed subset of $k_{\FF,>0}^n$, which is definable in $\mathfrak{k}_\fF$ (see \Cref{subsec:van_den_Dries} for the notation) and satisfies $\dim \initial_w(X) \le \dim X$.
\end{proposition}
\begin{proof}
    For any $a \in \FF_{>0}$, the set $a.X$ is a definable subset of $\FF_{>0}^n$ of the same dimension of $X$. Then, by \cite[Th.~A]{driesTConvexityTameExtensions1997}, $\res_\FF(a.X \cap (\cO_{\FF,>0}^\times)^n)$ is a closed definable subset of $k_{>0}^n$ of dimension at most $\dim a.X = \dim X$. The resul then follows by choosing $a = t^{-w}$.
\end{proof}
To our best knowledge, \Cref{lem:initial_closed_dimension} did not appear previously in the literature even for the semialgebraic setting.
We can describe initial degenerations using the angular component map, as follows.
\begin{lemma}
\label{Lemma:AlternativeDescrInitialDeg}
Let $\FF$ be a nontrivially valued real closed field and let $X\subset \FF_{>0}^n$ be a set. Then
    \[\initial_w(X) =  \ac_\FF (X\cap \val^{-1}(w) ).\]
\end{lemma}
\begin{proof}
    By definition~\eqref{Def:InitialDegeneration}, $y \in \initial_w(X)$ if and only if there exists $a \in t^{-w}.X \cap (\cO_{\FF,>0}^\times)^n$ such that $\res_\FF(a) = y$. This is equivalent to $t^w.a \in X$, $\val(t^w.a) = w$ and $\ac_\FF (t^w.a) = \res_\FF(t^{-\val(t^w.a)}.(t^w.a)) = \res_\FF(a)  = y$. 
\end{proof}
  
In the rest of this section, we describe the initial degeneration by an algebraic object, by the \emph{initial nonnegativity cone}, that plays the same role as initial ideals in tropicalizations of algebraic varieties.

The \emph{nonnegativity cone} of a set $X \subset \mathbb{F}^n_{>0}$ is
    \begin{align*}
    \mathcal{P}(X) \coloneqq \{\, f \in \mathbb{F}[\vb x^\Lambda] \mid f(a) \geq 0 \text{ for all } a \in X \,\}.
    \end{align*}
The nonnegativity cone $\mathcal{P}(X)$ is closed under addition, multiplication and contains all the squares of polynomials in $\mathbb{F}[\vb x^\Lambda]$, that is, it is a preordering \cite[Def.~1.1.23]{Scheiderer2024}. The following elementary fact about nonnegativity cones will be used repeatedly.
\begin{lemma}
\label{lemma:nonnegcones}
    Let $\mathbb{K}$ be a real closed field and $X \subset \mathbb{K}^n_{>0}$ be a set with nonnegativity cone $\mathcal{P}(X) \subset \mathbb{K}[\vb x^\Lambda]$. Then it holds that
    \begin{align}
    \label{eq:lemma:nonnegcones}
    X \subset \{\, a \in \mathbb{K}^n_{>0} \mid f(a) \geq 0 \text{ for all } f \in \mathcal{P}(X)\,\}.
    \end{align}
    Moreover, if $X$ is closed in $\KK^n_{>0}$, then equality holds in~\eqref{eq:lemma:nonnegcones}.
\end{lemma}
\begin{proof}
    Let $T = \{\, a \in \mathbb{K}^n_{>0} \mid f(a) \geq 0 \text{ for all } f \in \mathcal{P}(X)\,\}$.  The inclusion $X\subset T$ follows by definition. Now assume that $X$ is closed and let $a \in \KK^n_{>0} \setminus X$. Since $X$ is closed and by the definition of strong topology, there exists $r \in \FF_{>0}$ such that for the closed ball with center $a$ and radius $r$ we have
    \[ \{ x \in \KK^n_{>0} \mid (x_1 - a_1)^2 + \dots + (x_n - a_n)^2 - r^2 \leq 0\} \subset \KK^n_{>0} \setminus X \]
    Thus, $(x_1 - a_1)^2 + \dots + (x_n - a_n)^2 - r^2$ is negative at $a$ but positive on $X$, showing that $\KK^n_{>0} \setminus X \subset \KK^n_{>0} \setminus T$ and concluding the proof.
\end{proof}

\begin{proposition}
\label{Prop:InitialDegViaPreorder}
Let $\FF$ be a nontrivially valued real closed field.
For any set $X \subset \FF_{>0}^n$, we have that
    \begin{align}
\label{Eq:InitialSetAndInitialNonnegCone}
        \initial_w(X) \subset \{ \, y \in k_{\FF,>0}^n \mid  \initial_w(f)(y) \ge 0  \text{ for all } f\in \pos(X)\, \}
    \end{align}
     Moreover, if $X$ is closed in $\FF^n_{>0}$, then we have equality in~\eqref{Eq:InitialSetAndInitialNonnegCone}. 
\end{proposition}

\begin{proof}
    The splitting $t > 0$ is positive by assumption, and the residue map $\res_{\mathbb{F}}$ is compatible with the ordering on $\mathbb{F}$ and $k_\mathbb{F}$. Then for all $f \in \mathbb{F}[\vb x^\Lambda]$ and $b\in (\mathcal{O}_{\mathbb{F}}^\times)^n$, using the description of initial forms from Lemma~\ref{prop:InitilFormPolynomials} we have that
 \begin{align}
 \label{eq:compatibleorders}
      f(t^{w}.b) \geq 0 \quad \text{if and only if} \quad \initial_w(f)(\res_\FF(b)) \geq 0.
 \end{align}

 By definition, $y \in \initial_w(X)$ if and only if there exists $b\in (\mathcal{O}_{\mathbb{F}}^\times)^n$ such that $t^{w}.b \in X$ and $\res_{\mathbb{F}}(b) = y$. From Lemma~\ref{lemma:nonnegcones}, it follows that $f(t^{w}.b) \geq 0$ for all $f \in \mathcal{P}(X)$, and using \eqref{eq:compatibleorders}, we have that $y$ is contained in the right-hand side of~\eqref{Eq:InitialSetAndInitialNonnegCone}.

 Now let $X$ be closed and let $y \in k_{\FF,>0}^n$ satisfy $\initial_w(f)(y) \geq 0$ for all $f \in \mathcal{P}(X)$. By \eqref{eq:compatibleorders} we have $f(t^w.b) \geq 0$ for any $b\in (\mathcal{O}_{\mathbb{F}}^\times)^n$ with $\res_{\mathbb{F}}(b) = y$, and using Lemma~\ref{lemma:nonnegcones} we have $t^w.b \in X$, concluding the proof.
\end{proof}

\begin{remark}
\label{rem:initial}
In the special case where $X$ is an algebraic variety intersected with $\FF^n_{>0}$, one might be tempted to believe that it is enough to consider the initial forms of polynomials vanishing on $X$ in order to obtain the equality in Proposition~\ref{Prop:InitialDegViaPreorder}. However, this requires additional smoothness assumptions on $X$, see \cite[Th. 3.8]{RoseTelek} or \cite[Lemma 4.33]{Rau2025}. An example showing that these assumptions are necessary can be realized using a similar construction as in \Cref{Ex:SmoothConditions}. We refer to  \cite[Rem.~3.9]{RoseTelek} for more details.
\end{remark}

\subsection{First version of the Fundamental Theorem}
\label{Sec:FristVersion}
In this section, we prove a version of the Fundamental Theorem for both the tropicalization and the fine tropicalization of definable sets.
The first equality in \eqref{Eq:WeakFundamentalFineTrop} of \Cref{Thm:WeakFundamentalFineTrop} is the real counterpart of \cite[Th.~5.2]{maxwellGeometryTropicalExtensions2024}.

\begin{theorem}
\label{Thm:WeakFundamentalFineTrop}
Let $\FF$ be a nontrivially valued real closed field.
For any nonempty $X \subset \FF_{>0}^n$, we have 
    \begin{align}
    \label{Eq:WeakFundamentalFineTrop}
        \rv(X) =  \bigcup_{ w \in \Gamma_\FF^n} \initial_w(X) \times \{w\}  \subset \bigcap_{f \in \mathcal{P}(X)}     \{ \, z \in \RV_{\FF,>0}^n \mid \trop_{\rv}(f)(z) \geq 0 \,\}.
        \end{align}
        If $X$ is closed in $\FF^n_{>0}$, then equality holds in~\eqref{Eq:WeakFundamentalFineTrop}.
\end{theorem}
\begin{proof}
    A pair $(y,w)$ lies in $\rv(X) \subset k^n_{\FF,>0} \times \Gamma^n_\FF$ if and only if there exists $a \in X$ such that $\ac_\FF(a) = y$ and $\val(y) =w$, which is equivalent to $y \in \ac_\FF(X \cap \val^{-1}(w) )$. Thus, the first equality in~\eqref{Eq:WeakFundamentalFineTrop} follows from Lemma~\ref{Lemma:AlternativeDescrInitialDeg}. 

    By Proposition~\ref{Prop:InitialDegViaPreorder}, we have $\initial_w(X) \subset \bigcap_{f \in \mathcal{P}(X)} \{\, y  \in k_{\FF,>0}^n \mid \initial_w(f)(y) \geq 0 \, \}$ with equality if $X$ is closed. Therefore
    \begin{align*}
       \bigcup_{ w \in \Gamma_\FF^n} \initial_w(X) \times \{w\}  &\subset  \bigcup_{ w \in \Gamma_\FF^n}   \left( \bigcap_{f \in \mathcal{P}(X)} \{y  \in k_{\FF,>0}^n \mid \initial_w(f)(y) \geq 0 \} \right) \times \{w\} \\
       &= \left\{  (y,w) \in \Gamma_\FF^n \times k_{\FF,>0}^n \mid \initial_w(f)(y) \geq 0  \text{ for all } f \in \mathcal{P}(X) \right\} \\
        &\overset{(*)}{=}     \bigcap_{f \in \mathcal{P}(X)}  \{\,(y,w)  \in \RV_{\FF,>0}^n \mid \trop_{\rv}(f)(y,w) \geq 0 \,\}
    \end{align*}
    where in (*) we used that  $\initial_w(f)(y) \geq 0$ if and only if  $\trop_{\rv}(f)(y,w) \geq 0$, see Lemma~\ref{Eq:TropPrecludingPosValues}.
\end{proof}

\begin{corollary}
\label{Cor:WeakFundamentalTrop}
Let $\FF$ be a nontrivially valued real closed field.
For any nonempty subset $X \subset \FF_{>0}^n$, we have the following inclusions of sets
    \begin{align}
    \label{Eq:Cor:WeakFundamentalTrop}
        \val(X) = \{ \, w \in \Gamma^n_\mathbb{F} \mid \initial_w(X) \neq \emptyset \, \} \subset \bigcap_{f \in \mathcal{P}(X)}     \{ w \in \Gamma_\mathbb{F}^n \mid \trop_r(f)(1,w) \geq 0\}.
        \end{align}
\end{corollary}

\begin{proof}
    Since the diagram \eqref{Eq:CommutativeDiagram} commutes we have $\val_{\rv}(\rv(X)) = \val(X)$. The equality 
    \[\val_{\rv}\big( \bigcup_{ w \in \Gamma^n} \initial_w(X) \times \{w\}  \big) =  \{ \, w \in \Gamma^n_\mathbb{F} \mid \initial_w(X) \neq \emptyset \, \} \] is also immediate. 
    Thus, the equality in~\eqref{Eq:Cor:WeakFundamentalTrop} follows from Theorem~\ref{Thm:WeakFundamentalFineTrop}.
    To show the remaining inclusion, let $w \in \Gamma_\FF^n$ such that $\initial_w(X) \neq \emptyset$. By Theorem~\ref{Thm:WeakFundamentalFineTrop}, there exists $z = (y,w) \in \RV_{\FF,>0}^n$ such that $\trop_{\rv}(f)(z) \geq 0$ for all $f \in \mathcal{P}(X)$. Lemma~\ref{Lemma:TropPrecludingPosValues} implies that  $\trop_r(f)(1,w) \geq 0$ for all $f \in \mathcal{P}(X)$. Thus $w = \val_{\rv}(y,w)$ is contained in the right-hand side of~\eqref{Eq:Cor:WeakFundamentalTrop}.    
\end{proof}

In \Cref{Thm:WeakFundamentalTrop}, we strengthen \Cref{Cor:WeakFundamentalTrop} for definable sets in two ways. First, we consider the case where $X \subset \KK^n_{>0}$ and $\KK$ is a real closed field whose valuation may be trivial. Second, we show that the intersection may be taken over finitely many polynomials.

Our proof closely follows \cite[Sec.~6]{jellRealTropicalizationAnalytification2020}. However, for completeness, we include the necessary arguments here. To this end, we need the following results.

\begin{definition}
\label{Def:LambdaGammaPolyhedra}
Let $\Lambda$ and $\Gamma$ be subgroups of $\mathbb{R}$. Generalizing \cite[Def.~6.5]{jellRealTropicalizationAnalytification2020}, we define a \emph{$\Gamma$-affine $\Lambda$-polyhedron} in $\mathbb{R}^n$ as a set of the form
\[\left\{\,p \in \mathbb{R}^n \mid \lambda_i \cdot p +w_i \leq 0 \text { for } i \in [m] \,\right\},
\]
and an \emph{open $\Gamma$-affine $\Lambda$-polyhedron} in $\mathbb{R}^n$ as a set of the form 
\[
\left\{\,p \in \mathbb{R}^n \mid \lambda_i \cdot p+w_i<0 \text { for } i\in [ m] \,\right\}
\]
where $\lambda_1, \ldots, \lambda_m \in \Lambda^n$, and $w_1, \ldots, w_m \in \Gamma$.
\end{definition}
 When $\Gamma$ is the value group of a real closed field, which is an ordered vector space over $\Lambda$, the field of exponents of a polynomially bounded o-minimal structure, then open and closed $\Gamma$-affine $\Lambda$-polyhedra are extension to $\RR$ of semilinear sets (see \Cref{subsec:van_den_Dries}).

\begin{proposition}
\label{Prop:TropPolyhedralStructure}
    Let $\fK \prec \fF$ and $\KK \subset \FF$ be as in Assumption~\ref{ass:elementary_extensions}.
    For any definable set $X \subset \KK^n_{>0}$, the closure of $\val(X_{\fF})$ in $\mathbb{R}^n$, i.e. $\Trop(X)$, is a finite union of $\Gamma_{\KK}$-affine $\Lambda$-polyhedra.
\end{proposition}

\begin{proof}
The proof follows by combining the polynomially bounded o-minimal preparation theorem (see \Cref{thm:preparation_theorem}) with the o-minimal cell decomposition theorem (see \cite[p.~52]{driesTameTopologyOminimal1998} or \cite[Th.~2.10]{costeIntroductionOminimalGeometry1999}). A \emph{cell} is a definable set inductively defined as follows. For $\ell=1$, a cell is simply a singleton or an interval in $\KK = \KK^1$. For $\ell>1$, a cell in $\KK^\ell$ can be the graph of a continuous definable function $h \colon D \subset \KK^{\ell-1} \to \FF$, where $D$ is a cell in $\KK^{\ell-1}$, or the set of points lying between the graphs of two continuous definable function $f_1,f_2 \colon D \subset \KK^{\ell-1} \to \FF$ satisfying $f_1<f_2$ on the cell $D$. The cell decomposition theorem asserts that every definable set can be written as disjoint union of cells.

Let $f_1,f_2$ be definable functions as above appearing in a cell decomposition of $X \subset \KK^n_{>0}$. Refining the cell decomposition if necessary, we can assume that $f_1,f_2$ can be written as in \Cref{thm:preparation_theorem}:
\[
    f_i(s,a) = \abs{s-r_i(a)}^{\mu_i} c_i(a) u_i(s,a)
\]
for all $(s,a) \in D \subset \KK^{n-1}$, where $\epsilon \in \QQ_{>0}$ is small, $\mu_i \in \Lambda = \Lambda(\fK)$, $r_i, c_i$ are continuous $(n-2)$-ary functions definable over $\KK$, and $u_i$ is a $(n-1)$-ary function definable over $\KK$ such that $1-\epsilon<u<1+\epsilon$ on $D$. The cell decomposition is compatible with field extensions, and we use the same notation for the extension of the functions to $\fF$. Since $1-\epsilon<u_i<1+\epsilon$ on $D_\fF$, and  we have that $\val_\FF(u_i(s,a)) = 0$ on $D_\fF$. Hence:
\[
    \val_\FF(f_i(s,a)) = \mu_i \val_\FF(s-r_i(a)) + \val_\FF(c_i(a))
\]
on $D_\fF$. We can proceed by induction on $n$ (since $r_i$ and $c_i$ both depend on $n-2$ variables), and get an expression of the form $\mu \cdot p + v$ for some $\mu \in \Lambda^{n}$ and $v \in \Gamma_{\KK}^n$ (notice that a $0$-ary function is just a constant, and since $c_i$ is defined over $\KK$, we eventually get $c_i \in \KK$ and $\val_\FF(c_i) = \val_\KK(c_i) \in \Gamma_\KK$). By taking valuation on both sides in $f_1 < f_2$, we thus get a formula of the form $\lambda \cdot p + w \le 0$ for some $\lambda \in \Lambda^n$ and $w \in \Gamma_\KK$. The same computation shows that cells defined as the graph of $h$ as above give formulas of the form $\lambda \cdot p + w = 0$. By taking unions of all such cells, we have that 
$\val(X_\fF)$ can be written as union of sets of the form
\begin{align}
\label{Eq:ProofTropPolyhedralStructure}
\left\{\,p \in \Gamma_\FF^n \mid \lambda_i \cdot p +w_i = 0 \text { for } i=1,2, \ldots, r \text{ and } \lambda_j \cdot p +w_j \le 0 \text { for } j=r+1,2, \ldots, m\,\right\},
\end{align}
for some $\lambda_1,\dots,\lambda_m\in \Lambda^n$ and $w_1,\dots,w_m \in \Gamma_{\KK}$.
Since $\Gamma_\FF^n$ is dense in $\mathbb{R}^n$, it follows that the closure of the sets in~\eqref{Eq:ProofTropPolyhedralStructure} are $\Gamma_\KK$-affine $\Lambda$-polyhedra.
\end{proof}

\begin{lemma}
\label{Lemma:ComplementFiniteUnion}
Let $T$ be a finite union of $\Gamma$-affine $\Lambda$-polyhedra in $\mathbb{R}^n$. Then the complement of $T$ is a finite union of open $\Gamma$-affine $\Lambda$-polyhedra.
\end{lemma}

\begin{proof}
The proof is identical to the proof of \cite[Lem.~6.7]{jellRealTropicalizationAnalytification2020}. Since the argument there does not rely on the defining linear equations of $T$ having coefficients in $\ZZ$, one may replace them by coefficients from $\Lambda$.
\end{proof}

\begin{lemma}
\label{Lemma:OpenNegTropPoly}
Let $\Lambda$ and $\Gamma$ be subgroups of $\RR$ and consider the real tropical hyperfield $\mathbb{RT} = \mathbb{S} \rtimes \RR$.
 For any open $\Gamma$-affine $\Lambda$-polyhedron $P \subset \RR^n$, there exists a polynomial $F \in \mathbb{S} \rtimes \Gamma[\overline{\vb w}^\Lambda] \subset  \mathbb{RT}[\overline{\vb w}^\Lambda]$ with $P = \{ w \in \RR^n \mid F(1,w) < 0 \}$.
\end{lemma}

\begin{proof}
Let $\lambda_1, \ldots, \lambda_m \in \Lambda^n$ and $v_1, \ldots, v_m \in \Gamma$ such that
$$
P = \left\{ \,w \in \mathbb{R}^n \mid \lambda_i \cdot w+v_i<0 \text { for } i \in [m] \,\right\} = \{ \,w \in \mathbb{R}^n \mid \max_{i\in [m]} (\lambda_i \cdot w+v_i)<0  \,\}.
$$
Set $v_0 = 0$, $\lambda_0 = 0, \varepsilon_0 = -1$ and $\varepsilon_i = 1$ for $i \in [m]$. Consider the polynomial $F = \bigoplus_{i=0}^m (\varepsilon_i,-v_i) \odot \overline{w}^{\;-\lambda_i}  \in \mathbb{RT}[ \overline{\vb w}^\Lambda]
$. Lemma~\ref{Lemma:TropExtensionSumMinimum} implies that
\begin{align*}
 F(1,w) =  \bigoplus_{i=0}^m (\varepsilon_i, -v_i -\lambda_i \cdot w)  =  \bigoplus_{\substack{i=0,\dots,m\\   -v_i -\lambda_i \cdot w =  \min_{j=0,\dots,m}(-v_j -\lambda_j \cdot w) }} (\varepsilon_i, -v_i -\lambda_i \cdot w).
\end{align*}
Therefore, we have $F(1,w) \cap (\{1\} \times \RR \cup \{0\}) = \emptyset$ if and only if 
$$-v_0 - \lambda_0 \cdot w = 0< \min_{i\in [m]}(-v_i -\lambda_i \cdot w) = -\max_{i\in [m]} (\lambda_i \cdot w+v_i),$$
finishing the proof. 
\end{proof}

\begin{lemma}
\label{Lemma:LiftingPolynomialInequalities}
    Let $\fK \prec \fF$ and $\KK \subset \FF$ be as in Assumption~\ref{ass:elementary_extensions}, and let $\Lambda = \Lambda(\fK)$ be the field of exponents.
    Let $X \subset \mathbb{K}^{n}_{>0}$ be a definable set and let $F  \in \mathbb{RT}[\overline{w}^\Lambda]$ be polynomial with coefficients in $\mathbb{S} \rtimes \Gamma_\KK$ such that $F(1,w) \geq 0$ for all $w \in \val(X_\fF)$. Then there exists a polynomial $f \in \mathcal{P}( X) \subset \KK[\vb x^\Lambda]$  such that $\trop_r(f) = F$.  
\end{lemma}

\begin{proof}
Distinguishing the signs of its coefficients, we write $F$ as  $\bigoplus_{i=1}^r (1,v_i) \odot \overline{w}^{\lambda_i} \oplus  \bigoplus_{j=r+1}^m (-1,u_j) \odot \overline{w}^{\lambda_j}$ for some $v_1,\dots,v_r,u_{r+1},\dots,u_m \in \Gamma_\KK$ and $\lambda_1,\dots,\lambda_m \in \Lambda$. Choose $a_1,\dots,a_r,b_{r+1},\dots,b_m \in \mathbb{K}_{>0}$ such that  $\val(a_i) = v_i$, $\val(b_j) = u_j$ for each $i=1,\dots,r, \, j=r+1,\dots m$.

Following the proof of \cite[Lem.~6.8]{jellRealTropicalizationAnalytification2020}, we consider the definable map
\[\varphi \colon X \to \mathbb{K}, \quad x \mapsto \tfrac{\sum_{i=1}^r a_i x^{\lambda_i}}{\sum_{j=r+1}^m b_j x^{\lambda_j} },\]
and its extension $\varphi_\fF \colon X_\fF \to \mathbb{F}$. Since $X_\fF \subset \FF_{>0}^n$ and $a_i,b_j > 0$, we have $\varphi_\fF(x) > 0$ for all $x \in X_\fF$, 

If for some $x\in X_\fF$ we have $0 < \val( \varphi_\fF(x)) = \val\big( \sum_{i=1}^r a_i x^{\lambda_i}\big) - \val\big(\sum_{j=r+1}^m b_j x^{\lambda_j} \big)$, then the ultrametric inequality and Lemma~\ref{Lemma:TropExtensionSumMinimum} imply that
\begin{align*}
F(1,\val(x)) &=  \bigoplus_{i=1}^r (1,v_i) \odot (1,\val(x))^{\lambda_i} \oplus  \bigoplus_{j=r+1}^m (-1,u_j) \odot (1,\val(x))^{\lambda_j} \\
&=\bigoplus_{i=1}^r \sval(a_i x^{\lambda_i}) \oplus  \bigoplus_{j=r+1}^m \sval(-b_j x^{\lambda_j}) = \bigoplus_{j=r+1}^m \sval(-b_j x^{\lambda_j})  < 0,
\end{align*}
which contradicts our assumption $F(1,\val(x)) \geq 0$.

Thus, we have $0 \geq \val( \varphi_\fF(x)) $ for all $x \in X_\fF$.
If $0 > \val( \varphi_\fF(x))$ for all $x \in X_\fF$, then by order compatibility of the valuation map, cf.~\eqref{Eq:OrderCompatibleValuation}, we have $\varphi_\fF(x) > \varepsilon$ for any $\varepsilon\in \KK_{>0}$ with $\val(\varepsilon) = 0$. Therefore the polynomial
\begin{align}
\label{Eq:Def:Fepsilon}
f_\varepsilon(x) \coloneqq \sum_{i=1}^r a_i x^{\lambda_i} - \varepsilon \sum_{j=r+1}^m b_j x^{\lambda_j} \in \mathbb{K}[\vb x^\Lambda]
\end{align}
is nonnegative on $X_\fF$ and $\trop_r(f_\varepsilon) = F$.

In the rest of the proof, we consider the case when there exists $x_* \in X_\fF$ with  $\val( \varphi_\fF(x_*)) = 0$.
Let $\alpha \in \KK$ be the infimum of $\varphi$ on $X$.
Since the statement that $\alpha$ is the infimum of $\varphi$ on $X$ is expressible using formulas in  the language of $\fK$, and since $\fK \prec \fF$, it follows that $\alpha$ is also the infimum of $\varphi_\fF$ over $X_\fF$.

By order compatibility of the valuation map, from $\val(\varphi_\fF(X_\fF)) \leq 0 < 1 = \val(t)$ it follows that $0 < t < \varphi_\fF(X_\fF)$. In particular, $\alpha > 0$. Moreover, we have that $\val(\alpha) = 0$. To justify this claim, suppose that $\val(\alpha) \neq 0$. 

If $\val(\alpha)  < 0 = \val(\varphi_\fF(x_*))$, then $\alpha > \varphi_\fF(x_*)$ by order compatibility of the valuation map. Thus, we must have  $\val(\alpha)  > 0$. But then for $\alpha' \coloneqq t^{-\tfrac{1}{N}} \alpha$ we have $\val(\alpha') = \val(\alpha) - \tfrac{1}{N} > 0$ for some large enough $N \in \mathbb{N}$. In particular, $\alpha' < \varphi_\fF(x)$ for all $x \in X_\fF$ and $\alpha < \alpha'$, which is a contradiction.
Now, set $\varepsilon = \alpha$, and consider $f_\varepsilon$ as in \eqref{Eq:Def:Fepsilon}, for which we have that $f_\varepsilon \in \mathcal{P}(X_\fF) \cap \KK[\vb x^\Lambda]$ and $\trop_r(f_\varepsilon) = F$.
To finish the proof, we notice that {$\mathcal{P}(X_\fF) \cap \KK[\vb x^\Lambda] = \mathcal{P}(X)$.} Indeed, for $f \in \KK[\vb x^\Lambda]$, we have $\fK \models \forall x\in X\colon f(x) \ge 0$ if and only if $\fF \models  \forall x\in X_\fF\colon f_\fF(x) \ge 0$ by definition of elementary extension.
\end{proof}

The next lemma will be used not only in the proof of \Cref{Thm:WeakFundamentalTrop}, but also in the proof of \Cref{thm:fundamental_theorem}.

\begin{lemma}
\label{Lemma:WeakFundamentalTropOneInclusion}
 Let $\fK \prec \fF$ and $\KK \subset \FF$ be as in Assumption~\ref{ass:elementary_extensions}. For any definable subset $X \subset \KK^n_{>0}$, there exist polynomials $f_1,\dots,f_k \in \mathcal{P}(X)$ such that
    \begin{align*}
 \overline{\val(X_\fF)} = \bigcap_{i=1}^k     \{ w \in \RR^n \mid \trop_r(f_i)(1,w) \geq 0\}.
     \end{align*}
\end{lemma}

\begin{proof}
    To prove the equality, we adapt the argument showing $(3) \subset (4)$ in \cite[Th.~6.9]{jellRealTropicalizationAnalytification2020}. 
    By Proposition~\ref{Prop:TropPolyhedralStructure},
    $\overline{\val(X_\fF)}$ is a finite union of $\Gamma_\mathbb{K}$-affine $\Lambda$-polyhedra in $\RR^n$. 
   From Lemma~\ref{Lemma:ComplementFiniteUnion}, it follows that
    \[ \overline{\val(X_\fF)} = \RR^n \setminus \left(\cup_{i=1}^k P_i\right) \]
for some open $\Gamma_\KK$-affine $\Lambda$-polyhedra $P_1,\dots,P_k \subset \RR^n$.
By Lemma~\ref{Lemma:OpenNegTropPoly} there exist real tropical polynomials $F_1,\dots,F_k \in \mathbb{S} \rtimes \Gamma_\KK [\overline{\vb w}^\Lambda]$ such that $P_i = \{ w \in \RR^n \mid F_i(1,w) < 0 \}$ for each $i \in [k]$, which implies that 
\[ \overline{\val(X_\fF)} =  \bigcap_{i=1}^k\{ \, w \in \RR^n \mid F_i(1,w) \geq 0 \, \}. \]
By Lemma~\ref{Lemma:LiftingPolynomialInequalities}, there exist polynomials $f_1,\dots,f_k \in  \mathcal{P}(X)$ with $\trop_r(f_i) = F_i$ for $i \in [k]$, finishing the proof.
\end{proof}

With these technical lemmas in place, we are now ready to prove the main result of this section, which improves Corollary~\ref{Cor:WeakFundamentalTrop}.

\begin{theorem}
\label{Thm:WeakFundamentalTrop}
 Let $\fK \prec \fF$ and $\KK \subset \FF$ be as in Assumption~\ref{ass:elementary_extensions}.
 For a definable subset $X \subset \KK^n_{>0}$, the following sets are equal
    \begin{align*}
 \val(X_\fF) = \{ \, w \in \Gamma^n_\mathbb{F} \mid \initial_w(X_\fF) \neq \emptyset \, \} = \bigcap_{f \in \mathcal{P}(X)}     \{ w \in \Gamma_\mathbb{F}^n \mid \trop_r(f)(1,w) \geq 0\},
     \end{align*}
        and the intersection can be taken over finitely many polynomials.
\end{theorem}

\begin{proof}
The first equality holds by Corollary~\ref{Cor:WeakFundamentalTrop}. 
By Lemma~\ref{Lemma:WeakFundamentalTropOneInclusion}, there exist polynomials $f_1,\dots,f_k \in  \mathcal{P}(X)$ such that
\begin{align*}
 \val(X_\fF) &\overset{(1)}{=} \overline{ \val(X_\fF)} \cap \Gamma_\FF^n =  \bigcap_{i=1}^k  \{ w\in \Gamma_\FF^n \mid \trop_r(f_i)(1,w) \geq 0\}  \supset \bigcap_{f \in \mathcal{P}(X)}     \{  w\in \Gamma_\FF^n \mid \trop_r(f)(1,w) \geq 0 \}  \\
& \supset \bigcap_{f \in \mathcal{P}(X_\fF)}     \{  w\in \Gamma_\FF^n \mid \trop_r(f)(1,w) \geq 0 \} \supset \val(X_\fF),
\end{align*}
where $(1)$ holds by Corollary~\ref{cor:ValImageClosure} and the last inclusion follows from Corollary~\ref{Cor:WeakFundamentalTrop}.
\end{proof}

\section{Tropicalizing points in the o-minimal spectrum}
\label{Sec:Sec4}
\subsection{The o-minimal spectrum}
\label{Sec:OminimalSpectrum}
We recall in this section some fundamental properties of the o-minimal spectrum, following \cite{carralNormalSpectralSpaces1983, pillaySheavesContinuousDefinable1988, millerGrowthDichotomyOminimal1996, Coste1998, tresslRealSpectrumContinuous1997, edmundoSHEAFCOHOMOLOGYOMINIMAL2006, baroSpectralSpacesOminimal2024}. We chose to represent the points in the o-minimal spectrum as ultrafilters instead of complete types, as ultrafilters are more commonly used in real algebraic geometry \cite{brockerRealSpectraDistributions1982,brumfielUltrafilterTheoremReal1989,costeIntroductionOminimalGeometry1999,prestelPositivePolynomialsHilberts2001}.
\begin{definition}
    Let $Z$ be a nonempty set and $\mathcal{D} \subset 2^Z$ be a Boolean algebra of subsets of $Z$, partially ordered by inclusion. A \emph{filter} $\alpha$ on $\mathcal{D}$ is a nonempty collection of subsets in $\mathcal{D}$ satisfying the following properties:
    \begin{enumerate}
        \item $\emptyset \notin \alpha$;
        \item if $A \in \alpha$ and $B \in \mathcal{D}$ with $A\subset B$, then $B \in \alpha$; and 
        \item if $A,B \in \alpha$ then $A \cap B \in \alpha$.
    \end{enumerate}
    We say that a filter $\alpha$ is an \emph{ultrafilter} if, in addition, it satisfies the following property:
    \begin{enumerate}
        \item[(iv)] for each $A\in \mathcal{D}$, if $A \notin \alpha$, then $Z \setminus A \in \alpha$.
    \end{enumerate}
\end{definition}

\begin{definition}
    Let $\fK$ be an o-minimal expansion of a real closed field $\KK$, and let $\Def(X)$ be the Boolean algebra of definable subsets of a definable set $X \subset \KK^n$. The \emph{o-minimal spectrum}, denoted $\widetilde{X},$ is the set of ultrafilters on $\Def(X)$. We equip $\widetilde{X}$ with the \emph{spectral topology}, which has as basis the set of $\widetilde{U} = \{ \, \alpha \in \widetilde{X} \mid U \in \alpha \, \}$ for $U \subset X$ definable open; and with a sheaf $\cO_{\widetilde{X}}$, defined as
        $\cO_{\widetilde{X}}(\widetilde{U}) \coloneqq \cC^{\definable}(U)$,
    where $U \subset X$ is definable open and $\cC^{\definable}(U)$ is the ring of continuous definable functions from $U$ to $\KK$.
\end{definition}
If $f \colon X \to Y$ is a definable map with graph $G_f$, there is a natural extension $\widetilde{f} \colon \widetilde{X} \to \widetilde{Y}$ which has graph $\widetilde{G_f}$.

We recall an alternative, more geometric characterization of the o-minimal spectrum in the polynomially bounded case for locally closed sets, which follows from the \L ojasiewicz inequality. For a definable set $X \subset \KK$ and $f \in \cC^{\definable}(X)$, we denote by $\cZ_X(f) \coloneqq \{ \, a \in X \mid f(a) = 0\, \}$ the zero set of $f$ in $X$.

\begin{theorem}[{\cite[Th.~8.1]{tresslRealSpectrumContinuous1997}; see also \cite[Fact~3.5]{baroSpectralSpacesOminimal2024}}]
\label{Thm:OminimalSpectrum}
    Let $\fK$ be a polynomially bounded o-minimal expansion of a real closed field $\KK$, and let $X \subset \KK^n$ be a locally closed definable set. Then the maps
    \begin{align*}
    \widetilde{X} & \longrightarrow \Spec \cC^{\definable}(X) & \Spec \cC^{\definable}(X) & \longrightarrow \widetilde{X}  \\ \alpha & \longmapsto \p_\alpha \coloneqq \{ \, f \colon \cZ_X(f) \in \alpha \, \}
       &\p & \longmapsto \alpha_\p \coloneqq \{ \, \cZ_X(f) \colon f \in \p \, \} 
\end{align*}
    define an isomorphism between the locally ringed spaces $\widetilde{X}$ and $\Spec \cC^{\definable}(X)$.
\end{theorem}

 In the proofs of Theorems~\ref{Thm:A} and~\ref{Thm:B}, we will replace $X$ by its closure without loss of generality, using \Cref{lem:closure}. Therefore, assuming that $X$ is locally closed is not restrictive for our purposes, and it enables us to apply \Cref{Thm:OminimalSpectrum}.
 
 We now discuss the local properties of $\widetilde{X}$, assuming $X$ is locally closed. Denote the stalk at $\alpha \in \widetilde{X}$ as
\[
    \cO_{\widetilde{X},\alpha} = \varinjlim_{\widetilde{U} \ni \alpha} \cO_{\widetilde{X}}(\widetilde{U}) = \varinjlim_{U \in \alpha} {\cC^{\definable}}({U}) \cong \cC^{\definable}(X)_{\p_\alpha}
\]
(where the direct limits are taken for $U \subset {X}$ open) and the residue field as
\[
\KK(\alpha) = \cO_{\widetilde{X},\alpha} / {\p_\alpha}\cO_{\widetilde{X},\alpha} \cong \quot ( \cC^{\definable}(X) / \p_\alpha)
\]
So far, $\KK(\alpha)$ is just a field. We now recall the structure {$\fK(\alpha)$ on $\KK(\alpha)$, called the \emph{definable ultrapower} of $\KK$ at~$\alpha$, which is an elementary extension of $\fK$}.
We provide in the following some details of the construction, and refer to \cite{pillaySheavesContinuousDefinable1988,tresslRealSpectrumContinuous1997, Coste1998, costeIntroductionOminimalGeometry1999} for related discussions. The reader not interested in the model-theoretic details can take \eqref{Def:Kalpha_Formulas} as the definition of the structure.

For $D_1, D_2 \in \alpha $ and $f_i \colon D_i \to \KK$ definable, declare $f_1 \sim f_2$ if and only if $\{ \, a \in D_1 \cap D_2 \mid f_1(a) = f_2(a) \, \} \in \alpha$. One can restrict to the case of continuous $f_1$ and $f_2$ by restricting their domains. The fact that this construction gives exactly the field $\KK(\alpha)$ defined before was shown in \cite[Cor.~3.3]{pillaySheavesContinuousDefinable1988}. Therefore:
\[
    \KK(\alpha) \cong \varinjlim_{D \in \alpha} {\cC}^{\definable}(D) 
\]
Identifying any definable function $f$
with the infinite tuple $(f(a))_{a\in \KK^n}$ (setting $f(a) = 0$ if $a$ is not in the domain of $f$), we can also construct $\KK(\alpha)$ as the underlying set of the boolean ultrapower of $\cL$-structures:
\[
    \fK(\alpha) \coloneqq \prod_{a\in \KK^n} \fK \big/ \alpha
\]
see \cite{mansfieldTheoryBooleanUltrapowers1971} for the precise definition, and also \cite{prestelMathematicalLogicModel2011} for more details on ultrapowers and ultraproducts. \L os's theorm for boolean ultrapowers \cite[Th.~1.5]{mansfieldTheoryBooleanUltrapowers1971} (see also \cite[Th.~2.6.2]{prestelMathematicalLogicModel2011}) implies that, for every formula $\phi$ and every $f(\alpha) \in \KK(\alpha)$:
\begin{align}
\label{Def:Kalpha_Formulas}
    \fK(\alpha) \models \phi(f(\alpha)) \Longleftrightarrow \exists D \in \alpha \text{ s.t. } \fK \models \forall a \in D \ \phi(f(a))
\end{align}
In particular, this implies that the extension $\fK \prec \fK(\alpha)$, where we identify every $a \in \KK$ with a constant function, is elementary \cite[Cor.~1.6]{mansfieldTheoryBooleanUltrapowers1971}.

For all $\alpha \in \widetilde{X}$, $\KK(\alpha)$ is a real closed field, see e.g. \cite{Coste1998, costeIntroductionOminimalGeometry1999}. This is an important step toward the third characterization of the o-minimal spectrum, which we now discuss.

\begin{definition}
    Let $A$ be a commutative ring with a unit. The \emph{real spectrum} of $A$ is the set of all pairs $\alpha = (\mathfrak{p}_\alpha,\leq_\alpha)$, where $\mathfrak{p}_\alpha \in \Spec A$ is a prime ideal and $\leq_\alpha$ is an ordering on the residue field $\quot(A/\p)$. We denote the real spectrum of $A$ by $\sper(A)$.
\end{definition}

We refer the reader to \cite{Scheiderer2024} for the basic properties of the real spectrum. For every $\p \in \Spec \cC^{\definable}(X)$ the residue field $\KK(\alpha_\p) = \quot(\cC^{\definable}(X)/ \p_\alpha)$ is real closed, and in particular it admits a unique ordering. Hence 
\[
    \sper \cC^{\definable}(X) \longrightarrow \Spec \cC^{\definable}(X), \quad (\p, \le) \longmapsto \p
\]
is a bijection. Moreover, it is proven in \cite[Lem.~2.4 and Th.~8.1]{tresslRealSpectrumContinuous1997} (see also \cite[Fact~3.1]{baroSpectralSpacesOminimal2024}) that, in the polynomially bouded case, this map is also a homeomorphism when we equip the real spectrum with the Harrison topology (see e.g. \cite{Scheiderer2024} for the definition of the Harrison topology).

We have then three different incarnations of $\widetilde{X}$ in the polynomially bounded case, for locally closed definable $X$:
\begin{enumerate}
    \item as the set of ultrafilters on $\Def(X)$, by definition;
    \item as the spectrum of the ring $\cC^{\definable}(X)$;
    \item as the real spectrum of the ring $\cC^{\definable}(X)$.
\end{enumerate}
When $X$ is not locally closed, we only have an embedding $\widetilde{X} \hookrightarrow \Spec \cC^{\definable}(X) \cong \Sper \cC^{\definable}(X)$, see \cite[Prop.~3.1]{tresslRealSpectrumContinuous1997} or \cite[Fact~3.4]{baroSpectralSpacesOminimal2024}.
As a first application of this identification, we describe the closed points of $\widetilde{X}$.

\begin{definition}
    Let $X \subset \KK^n$ be a definable set in an o-minimal expansion of a real closed field $\KK$. We denote by $\widetilde{X}^{\max}$ the set of closed points in $\widetilde{X}$ with respect to the spectral topology.
\end{definition}

\begin{definition}
    Let $(\KK,\leq)$ be an ordered field and $A \subset \KK$ be a subring. We say that $\KK$ is \emph{relatively archimedean over $A$} with respect to $\leq$ if, for every $b\in \KK$, there exists $a\in A$ with $\abs{b} \leq a$. We say that an ordered field is \emph{archimedean} if it is relatively archimedean over $\QQ$, or equivalently if it is a subfield of $\RR$ (by H\"older's embedding theorem \cite[Th.~2.1.10]{Knebusch2022}).
\end{definition}

\begin{lemma}
    \label{lem:max}
    Let $\fK$ be a polynomially bounded o-minimal expansion of a real closed field $\KK$, and let $X \subset \KK^n$ be definable locally closed set. Then the following are equivalent.
    \begin{enumerate}
        \item $\alpha \in \widetilde{X}^{\max}$;
        \item $\KK(\alpha)$ is relatively archimedean over $\cC^{\definable}(X) / \p_\alpha$.
    \end{enumerate}
\end{lemma}
\begin{proof}
    This follows from \cite[Prop.~3.6.17]{Scheiderer2024} combined with the homoemorphism $\widetilde{X} \cong \sper \cC^{\definable}(X)$.
\end{proof}
In the semialgebraic case, the o-minimal spectrum of a semialgebraic set coincides with the standard tilda operation, see \cite{brockerRealSpectraDistributions1982}. The homomorphism between the real spectrum costruction and the spectrum of the ring of continuous semialgebraic functions has been been first proven in \cite[Prop.~6]{carralNormalSpectralSpaces1983}.
\subsection{Relatively archimedean points}
\label{Sec:RelArchPoints}
We study relatively archimedean points in the o-minimal setting. This concept has been introduced in \cite{jellRealTropicalizationAnalytification2020} to study tropicalizations of semialgebraic sets. In the following, we fix $\fK$ a polynomially bounded o-minimal expansion of a real closed field $\KK$ and a definable set $X \subset \KK^n$.
\begin{definition}
    We say that $\alpha \in \widetilde{X}$ is \emph{relatively archimedean} if $\KK(\alpha)$ is relatively archimedean over $\KK$. We denote the set of relatively archimedean points as $\widetilde{X}^{\arch}$, and equip it with the relative spectral topology.
\end{definition}
\begin{remark}
    \label{rem:berkovic}
    It is shown in \cite[Theorem 3.17]{jellRealTropicalizationAnalytification2020} that $\widetilde{X}^{\arch}$ is homeomorphic, in the semialgebraic case, to the real Berkovich analytification of $X$. While an analog statement holds true also in our generalized setup, we prefer in the following to work directly inside the o-minimal spectrum, since our model-theoretic techniques are better expressed in this context.
\end{remark}
We use \eqref{Def:Kalpha_Formulas} to characterize relatively archimedean points in terms of continuous semialgebraic functions.
\begin{lemma}
    \label{lem:arch_char}
    Let $X$ be a definable locally closed set. Then $\alpha \in \widetilde{X}^\arch$ if and only if for every $D \in \alpha$ and $f \in \cC^{\definable}(X)$, there exists $D' \in \alpha$ and $c \in \KK$ such that $\abs{f(a)} \le c$ for all $a \in D\cap D'$.
\end{lemma}
\begin{proof}
    Every element of $\KK(\alpha)$ can be represented as $f(\alpha)$ for some $f \in \cC^{\definable}(D)$ and $D \in \alpha$. Then the result follows from \eqref{Def:Kalpha_Formulas}.
\end{proof}
\begin{example}
    \label{ex:ominimal_spectrum}
    For any polynomially bounded o-minimal structure $\fK$ expanding a real closed field $\KK$, we describe the o-minimal spectrum $\widetilde{\KK_{>0}}$. Since the definable subsets of $\KK$ are finite unions of points and intervals, these are precisely the semialgebraic subsets of $\KK$. Thus, combining \cite[Ex.~3.1.14]{Scheiderer2024} and \cite[Prop.~6]{carralNormalSpectralSpaces1983}, we have the following characterization of $\widetilde{\KK_{>0}}$. 
    A \emph{Dedekind cut} of $\KK_{>0}$ is a pair $(I, J)$ such that $\KK_{>0} = I \cup J$ and $I<J$. We say that a Dedekind cut is \emph{free} is it is not of the form $]0, a], ]a, \infty[$ or $]0, a[, [a, \infty[$ for some $a \in \KK_{>0}$. Then:
    \begin{itemize}
        \item $\widetilde{\KK_{>0}}^\arch$ contains exactly the set of \emph{principal} ultrafilters (i.e. the ultrafilters defined as the collection of definable subsets of $\KK_{>0}$ containing a given $a\in \KK_{>0}$) and one ultrafilter for each free Dedekind cut of $\KK_{>0}$.
        \item $\widetilde{\KK_{>0}}^{\max}$ is equal to the union of $\widetilde{X}^\arch$ with $0^+$ and $+\infty$, where $0^+$ (resp. $+\infty$) is the ultrafilter of definable subsets containing $]0,\epsilon]$ for some $\epsilon \in \KK_{>0}$ (resp. $[M, +\infty[$ for some $M \in \KK_{>0}$).
        \item $\widetilde{\KK_{>0}}$ is equal to the union of $\widetilde{X}^{\max}$ with $a^+$ and $a^-$ for $a \in \KK_{>0}$, where $a^+$ (resp. $a^-$) is the ultrafilter of definable sets containing $]a,a+\epsilon]$ (resp. $[a-\epsilon, a[$) for some $\epsilon \in \KK_{>0}$.
    \end{itemize}
    We denote $\cH(\fK) \coloneqq \KK(+\infty)$. This is the Hardy field of germs of continuous definable functions at infinity (isomorphic to the algebraic Puiseux series in the semialgebraic case, see \cite{BasuPollackRoy}), and the structure $\fH(\fK) \coloneqq \fK(+\infty)$ is an elementary extension of $\fK$. The field $\cH(\fK)$ admits a nontrivial valuation, which is zero on $\KK$, and whose value group is the field of exponents $\Lambda(\fK)$, and the splitting is given by the germ of the identity function. We refer to \cite{millerGrowthDichotomyOminimal1996} for more details.
\end{example}

Notice that the inclusion $\widetilde{X}^\arch \subset \widetilde{X}^{\max}$ follows from \Cref{lem:max}. We will show in \Cref{lem:closed_bounded} that equality holds if and only if $X$ is closed and bounded. To do this, we need first to show that every definable map $f \colon X \to Y$ descends to a map of relatively archimedean points.
\begin{lemma}
    \label{lem:archimedean_functor}
    Let $f \colon X \to Y$ be a definable map between locally closed definable sets, and $\alpha \in \widetilde{X}$. If $\alpha \in \widetilde{X}^\arch$, then $\widetilde{f}(\alpha) \in \widetilde{Y}^\arch$. Moreover, if $f$ is injective and $\widetilde{f}(\alpha) \in \widetilde{Y}^\arch$, then $\alpha \in \widetilde{X}^{\arch}$.
\end{lemma}
\begin{proof}
    Notice that $\widetilde{f}(\alpha) = \{ \, D' \in \Def(Y) \mid f^{-1}(D') \in \alpha \, \}$. Let $\alpha \in \widetilde{X}^\arch$. If $h \in \cC^{\definable}(E)$ for some $E \in \widetilde{f}(\alpha)$, then there exists $D \in \alpha$ and $c \in \KK$ such that $\abs{h \circ f} \le c$ on $D$. By definition, $f(D) \in \widetilde{f}(\alpha)$ and $\abs{h} \le c$ on $f(D)\cap E$. Hence $\tilde{f}(\alpha) \in \widetilde{Y}^\arch$ by \Cref{lem:arch_char}.

    Assume that $f$ is injective and $\widetilde{f}(\alpha) \in \widetilde{Y}^\arch$. Let $g \in \cC^{\definable}(D)$ for some $D \in \alpha$. The function $f^{-1}$ is definable, hence there exists $E \in \widetilde{f}(\alpha)$ and $c \in \KK$ such that $\abs{g \circ f^{-1}} \le c$ on $E$. Then $\abs{g} \le c$ on $f^{-1}(E) \cap D \in \alpha$, that implies $\alpha \in \widetilde{X}^\arch$ by \Cref{lem:arch_char}.
\end{proof}
\Cref{lem:archimedean_functor} shows that the diagram
\begin{center}
    \begin{tikzcd}[sep=large]
     X \arrow[r, "f"] \arrow[d, hook]
    & Y \arrow[d, hook]\\
     \widetilde{X}^{\arch} \arrow[r, "\widetilde{f}"] & \widetilde{Y}^{\arch}
    \end{tikzcd}
\end{center}
is well-defined and commutes.    
\begin{lemma}
    \label{lem:closed_bounded}
    The equality $\widetilde{X}^\arch = \widetilde{X}^{\max}$ holds if and only if the locally closed definable set $X$ is closed and bounded.
\end{lemma}
\begin{proof}
    The inclusion $\widetilde{X}^\arch \subset \widetilde{X}^{\max}$ always holds by \Cref{lem:max}.

    Assume that $X$ is closed and bounded. Then for every $f \in \cC^{\definable}(X)$, there exists $c \in \KK$ such that $\abs{f} \le c$ on $X$, see e.g. \cite[Th.~3.4]{costeIntroductionOminimalGeometry1999}. This means that $\cC^{\definable}(X)$ is relatively archimedean over $\KK$. Now let $\alpha \in \widetilde{X}^{\max}$. From \Cref{lem:max}, $\KK(\alpha)$ is relatively archimedean over $\cC^{\definable}(X)/\p_\alpha$, which is relatively archimedean over $\KK$ since $\cC^{\definable}(X)$ is so. Hence $\alpha \in \widetilde{X}^\arch$.
    
    Assume now that $X$ is not closed or not bounded. Then, from the semialgebraic curve selection lemma, see e.g. \cite[Th.~3.2]{costeIntroductionOminimalGeometry1999}, there exists a definable curve  $g \colon ]0,1] \to X$ such that $g(\epsilon)$ converges to $a \in \cl{X} \setminus X$ or $\lim \norm{g(\epsilon)} = \infty$ for $\varepsilon \to 0$. In both cases, the set $Y_\epsilon \coloneqq g(]0,\epsilon])$ is relatively closed in $X$ for all $\epsilon$. Let 
    \[
        \alpha = \{ \, D \in \Def(X) \mid Y_\epsilon {\subset} D \text{ for all } \epsilon \text{ small enough} \, \}
    \]
    be the ultrafilter associated to $g$. It is well-known that $\alpha \in \widetilde{X}^{\max} \setminus \widetilde{X}^\arch$, which concludes the proof. This fact is also used in the proof of \cite[Prop.~4.2]{jellRealTropicalizationAnalytification2020} in the semialgebraic setting. However, since we could not find a proof of this fact in the literature, we give a full argument in the next paragraph.
    
    Since $Y_\epsilon$ is {relatively} closed in $X$,  $\alpha$ is a closed point of $\widetilde{X}$ if and only if $\beta \coloneqq \alpha \cap Y_\epsilon$ is a closed point of~$\widetilde{Y_\epsilon}$. By the definable triviality for functions (see e.g. \cite[Th.~5.24]{costeIntroductionOminimalGeometry1999}), for $\epsilon$ sufficiently small the map $g$ is a definable homeomorphism between $Y_\epsilon$ and $]0, \epsilon]$. Hence $\widetilde{g} \colon \widetilde{]0, \epsilon]} \to \widetilde{Y_\epsilon}$ is a homeomorphism, and $\widetilde{g}^{-1}(\beta) = 0^+ \in \widetilde{]0, \epsilon]}^{\max} \setminus \widetilde{]0, \epsilon]}^\arch$, see \Cref{ex:ominimal_spectrum} for the notation. This implies that $\beta \in \widetilde{Y_\epsilon}^{\max} \setminus \widetilde{Y_\epsilon}^{\arch}$. From \Cref{lem:archimedean_functor}, applied to the inclusion $Y_\epsilon \hookrightarrow X$, we conclude that $\alpha \in \widetilde{X}^{\max} \setminus \widetilde{X}^\arch$. 
\end{proof}

We conclude this subsection with a result (\Cref{lem:surjectivity_field_extension}) about the behavior of relatively archimedean points under elementary extensions. This result will play a crucial role in the proof of \Cref{thm:fundamental_theorem}. To that end, we first need the following lemma.
\begin{lemma}
    \label{lem:archimedean_points_1}
    Let $\fK$ be  a polynomially bounded o-minimal structure expanding a real closed field $\KK$. Then 
    \[\widetilde{\KK^n}^\arch = \bigcup_{s \in \KK_{>0}} \widetilde{]-s, s[^n} \cap \widetilde{\KK^n}^{\max} = \bigcup_{s \in \KK_{>0}} \widetilde{[-s, s]^n}^{\max} = \bigcup_{s \in \KK_{>0}} \widetilde{[-s, s]^n}^{\arch}\]
    In particular, $\widetilde{\KK^n}^\arch$ is relatively spectral open in $\widetilde{\KK^n}^{\max}$. Moreover,
    \[\widetilde{\KK_{>0}^n}^\arch = \bigcup_{s \in \KK_{>0}} \widetilde{]s^{-1}, s[^n} \cap \widetilde{\KK^n}^{\max} = \bigcup_{s \in \KK_{>0}} \widetilde{[s^{-1}, s]^n}^{\max} = \bigcup_{s \in \KK_{>0}} \widetilde{[s^{-1}, s]^n}^{\arch}\]
    and  $\widetilde{\KK^n_{>0}}^\arch$ is relatively spectral open in $\widetilde{\KK^n_{>0}}^{\max}$.
\end{lemma}
\begin{proof}
    Let $\alpha \in \widetilde{X}^\arch \subset \widetilde{X}^{\max}$. By \Cref{lem:arch_char}, there exists $D \in \alpha$ and $s \in \KK$ such that that $\abs{x_i(a)} < s$ for all $a\in D$ and $i\in [n]$, i.e., $D \subset ]-s,s[^n$. Since $\alpha$ is an ultrafilter we have $]-s,s[^n \in \alpha$, proving the inclusion $\widetilde{\KK^n}^\arch \subset \bigcup_{s \in \KK} \widetilde{]-s, s[^n} \cap \widetilde{\KK^n}^{\max}$.

    The inclusion $\bigcup_{s \in \KK} \widetilde{]-s, s[^n} \cap \widetilde{\KK^n}^{\max} \subset \bigcup_{s \in \KK} \widetilde{[-s, s]^n}^{\max} $ follows from  $\widetilde{[-s,s]^n}\cap\widetilde{\KK^n}^{\max} = \widetilde{[-s,s]^n}^{\max}$, that is true because $\widetilde{[-s,s]}$ is a closed subset of $\widetilde{\KK_n}$. The equality $\bigcup_{s \in \KK} \widetilde{[-s, s]^n}^{\max} = \bigcup_{s \in \KK} \widetilde{[-s, s]^n}^{\arch}$ follows from \Cref{lem:closed_bounded}. For the inclusion $\bigcup_{s \in \KK} \widetilde{[-s, s]^n}^{\arch} \subset \widetilde{\KK^n}^\arch$, apply \Cref{lem:archimedean_functor} to the inclusion $[-s,s]^n \hookrightarrow \KK^n$.

    The proof of the second part is similar.
\end{proof}

\color{black}

\begin{proposition}
    \label{lem:surjectivity_field_extension}
    Let $\fF$ be a polynomially bounded o-minimal expansion of a real closed field $\FF$ and let $X \subset \FF^n_{>0}$ be a locally closed definable set. Let $\fM$ be an elementary extension of $\fF$ such that the underlying real closed field $\MM$ is relatively archimedean over $\FF$. The diagram
    \begin{center}
    \begin{tikzcd}[sep=large]
     \widetilde{X_{\fM}} \arrow[r, two heads]
    & \widetilde{X_{\fF}}\\
     \widetilde{X_{\fM}}^{\arch} \arrow[u, hook] \arrow[r, two heads] & \widetilde{X_{\fF}}^{\arch} \arrow[u, hook]
    \end{tikzcd}
    \end{center}
    where the horizontal maps are induced by the inclusion of rings $\cC^{\definable}(X_{\fF}) \hookrightarrow \cC^{\definable}(X_{\fM})$ (sending $f$ to $f_\fM$),
    is well-defined, commutes, and has surjective horizontal maps.
\end{proposition}
\begin{proof}
    We show that the top horizontal map 
    \[
       \widetilde{X_{\fM}} \cong \Spec \cC^{\definable}(X_{\fM})\longrightarrow \Spec \cC^{\definable}(X_{\fF}) \cong \widetilde{X_{\fF}}
    \]
    is surjective. Using the isomorphism with the spectrum of the ring of definable functions, we see that the map sends $\beta \in \widetilde{X_{\fM}}$ to the ultrafilter $ \beta\cap\fF \coloneqq\{ \, W \cap X_{\fF} \mid W \in \beta \text{ is $\FF$-definable} \, \}$.
    Now let $\alpha \in \widetilde{X_{\fF}}$, and notice that
    \begin{equation}
        \label{eq:iff_extension}
        \beta \cap \fF = \alpha \Longleftrightarrow \forall V \in \alpha, \ V_\fM \in \beta
    \end{equation}
    since for $D \in \Def(X)$ we have $D_\fM \cap \FF^n = D$. Consider the family $\beta_0 \coloneqq \{ \, V_\fM \in \Def(X_\fM) \mid V \in \alpha \, \}$. This family has the finite intersection property, i.e., any finite intersection of sets in $\beta_0$ is nonempty. As in \cite[Lem.~2.6.1]{prestelMathematicalLogicModel2011}, we can then construct an ultrafilter $\beta \in \widetilde{X_\fM}$ extending $\beta_0$. It follows from \eqref{eq:iff_extension} that $\beta \cap \fF = \alpha$.

    We now study the bottom horizontal map. We first need to show that it is well-defined, that is, for $\beta \in \widetilde{X_{\fM}}^{\arch}$, we have $ \alpha \coloneqq \beta\cap\fF \in \widetilde{X_{\fF}}^{\arch}$. Consider the commutative diagram
    \begin{center}
    \begin{tikzcd}[sep=large]
     \FF \arrow[r, hook] \arrow[d, hook]
    & \MM \arrow[d, hook]\\
     \FF(\alpha) \arrow[r, hook] & \MM(\beta)
    \end{tikzcd}
    \end{center}
    where the bottom map follows from the inclusion $\cC^{\definable}(X_{\fF})/\p_{\alpha} \hookrightarrow \cC^{\definable}(X_{\fM})/\p_{\beta}$. Since $\beta \in \widetilde{X_{\fM}}^{\arch}$, $\MM(\beta)$ is relatively archimedean over $\MM$, and by hypothesis $\MM$ is relatively archimedean over $\FF$. Therefore $\MM(\beta)$ is relatively archimedean over $\FF$, which implies that $\FF(\alpha)$ is relatively archimedean over $\FF$, i.e. $\alpha \in \widetilde{X_{\fF}}^{\arch}$.

    Now we prove that the map $\widetilde{X_{\fM}}^\arch \to \widetilde{X_{\fF}}^\arch$ is surjective. Let $\alpha \in \widetilde{X_{\fF}}^\arch$, and pick $\beta \in \widetilde{X_{\fM}}$ such that $\beta \cap \fF = \alpha$. Let $\overline{\beta}$ be the unique closed point in $\cl{\{\beta\}} \cap \widetilde{X_{\fM}}^{\max}$ (the uniqueness follows e.g. from \cite[Prop.~3.4.9]{Scheiderer2024} or \cite[Fact~3.2]{baroSpectralSpacesOminimal2024}).
    We first show that $\overline{\beta} \cap \fF = \beta \cap \fF = \alpha$. Notice that $\overline{\beta} \supset \beta$, and hence the inclusion $\overline{\beta} \cap \fF \supset \beta \cap \fF$ follows. But $\alpha \in \widetilde{X_{\fF}}^\arch \subset \widetilde{X_{\fF}}^{\max}$ is closed, and thus $\overline{\beta} \cap \fF$, a specialization of $\beta \cap \fF = \alpha$, is itself equal to $\alpha$.

    Since $\alpha \in \widetilde{X_{\fF}}^\arch$, for every $i \in [n]$ we have $c^{-1}\le x_i({\alpha}) \le c$ for some $c \in \FF$, which implies that $[c^{-1},c]^n \in \alpha$ from \eqref{Def:Kalpha_Formulas}. But then $[c^{-1},c]^n_{\fF} \in \overline{\beta}$ by \eqref{eq:iff_extension},
    i.e. $c^{-1}\le {x_i({\overline{\beta}})} \le c$ for every $i \in [n]$. Now, it follows from \Cref{lem:archimedean_functor,lem:archimedean_points_1} that $\overline{\beta} \in \widetilde{X_{\fM}}^{\arch}$.

    Finally, the commutativity of the diagram is obvious by definition.
\end{proof}

\subsection{Valuations and relatively archimedean points}
\label{Sec:RelArchPointsVal}
We now study natural valuations associated to the residue field of every relatively archimedean point, in the case the base field admits a nontrivial valuation. The next lemma is the analog to \cite[Lemma~3.13]{jellRealTropicalizationAnalytification2020} in our context.

\begin{lemma}
    \label{lem:valuation_archimedean}
    Assume that $\FF\subset \LL$ is an extension of real closed fields, that $\FF$ has an order compatible (rank one) valuation $\val_\FF \colon \FF \to \Gamma_\FF \cup \{ \infty \}$, and that $\LL$ is archimedean over $\FF$.
    Then there is a unique order compatible (rank one) valuation $\val_{\LL}$ extending $\val_{\FF}$, defined  as:
    \begin{align}
        \label{Eq:Lemma:UpperLowerBounds}
        \val_\LL(x) = \sup\{ \, \val_\FF(a) \mid a\in \FF, \ a \ge x  \,\} = \inf\{ \, \val_\FF(a) \mid a\in \FF, \ a \le x  \,\}
    \end{align}
    This is the valuation having as valuation ring the convex hull of $\cO_\FF$ in $\LL$, i.e. the elements in $\LL$ which are upper and lower bounded by elements in $\cO_{\FF}$.
\end{lemma}
\begin{proof}
    Denote $R$ the convex hull of $\cO_{\FF}$ in $\LL$. Since $R$ is a convex subring, then $R$ is a valuation ring \cite[Prop.~2.2.4]{Knebusch2022}. In other words, $R$ is the smallest convex valuation ring of $\LL$ containing $\cO_\FF$. 
    We show that the valuation {$v_R \colon \LL^\times \to \LL^\times / R^\times \eqqcolon \Gamma_R$} associated to $R$ is of rank one, or equivalently that $ \Gamma_R$ is archimedean (i.e. relatively archimedean over $\QQ$), or equivalently that $\Gamma_R$ is a subfield of $\RR$ \cite[Th.~2.1.10]{Knebusch2022}. Since $R \cap \FF = \cO_\FF$, we can embed $\Gamma_\FF = \FF^\times / \cO_\FF^\times \hookrightarrow \LL^\times / R^\times = \Gamma_R$ and $v_R |_{\FF} = \val_{\FF}$. Pick $\gamma = v_R(a) \in \Gamma_R$. Without loss of generality, assume $a > 0$. By archimedeanity, there exists $b \in \FF$ such that $a^{-1} \le b$, and hence $v_R(a) \le v_R(b^{-1}) = \val_{\FF}(b^{-1})$. Therefore, $\Gamma_R$ is archimedean over $\Gamma_\FF \subset \RR$, which implies that $\Gamma_R$ is itself archimedean, and a subfield of $\RR$.
    
    Now let $\val_\LL \colon \LL \to \Gamma_\LL \cup \{ \infty \}$ be any order compatible rank one valuation of $\LL$ extending $\val_\FF$. For $x \in \LL$, let $\mathrm{UB}_\FF(x)$ (resp. $\mathrm{LB}_\FF(x)$) the set of upper bounds (resp. lower bounds) of $x$ in $\FF$. Since $\LL$ is archimedean over $\FF$ both $\mathrm{UB}_\FF(x)$ and $\mathrm{LB}_\FF(x)$ are nonempty. The valuations are compatible with the ordering, and since $\Gamma_\FF \subset \RR$, we have
    \[
        \val_{\LL}(x) = \sup \val_{\FF}(\mathrm{UB}_\FF(x)) = \inf \val_{\FF}(\mathrm{LB}_\FF(x)) \in \RR
    \]
    for $0 \neq x \in \LL$. Hence any such valuation is uniquely determined by~\eqref{Eq:Lemma:UpperLowerBounds}, and in particular $\val_\LL = v_R$.
\end{proof}
\begin{definition}
\label{def:val_rv_alpha}
    Let $\val\colon \FF \to \Gamma_\FF \cup \{ \infty \}$ be a {nontrivial} order compatible valuation on $\FF$, and $X \subset \FF^n$ be definable in {a polynomially bounded o-minimial expansion of $\FF$}. If $\alpha \in \widetilde{X}^\arch$, we denote 
    \[
        \val_\alpha \coloneqq \val_{\FF(\alpha)} \colon \FF(\alpha) \to \Gamma(\alpha)
    \]
    the unique order compatible valuation on $\FF(\alpha)$ from \Cref{lem:valuation_archimedean}. We also denote 
    \[
        \rv_\alpha \colon \FF(\alpha)^\times \to \RV(\alpha)^\times = \FF(\alpha)^\times /(1+ \m_\alpha)
    \]
    the corresponding RV sort.
\end{definition}

\subsection{Tropicalization maps and strong density theorem}
    \label{sec:strong_density}
    Let $\fF$ be a polynomially bounded o-minimal expansion of a real closed field $\FF$, equipped with a nontrivial order compatible valuation $\val \colon \FF \to \Gamma_\FF \cup \{ \infty \}$, and let $X \subset \FF_{>0}^n$ be a definable set. In the following, we always assume that $\alpha \in \widetilde{X}^\arch$.

Let $k = k_\FF$ be the residue field of $\FF$ and $k(\alpha) \coloneqq k_{\FF(\alpha)}$ be the residue field of $\FF(\alpha)$. There exists a universal real closed field $\LL$ and $k_\FF$-embeddings $k(\alpha) \hookrightarrow \LL$. If $k$ is archimedean, then $k(\alpha)$ is archimedean as well,  and we can simply take $\LL = \RR$. In general, the field $\LL$ can be for instance constructed as follows.

Let $A = \bigotimes_{k} k(\alpha)$ be the coproduct of the family of $\FF$-algebras $\{ \, k(\alpha) \colon \alpha \in \widetilde{X}^\arch \, \}$. This can also be interpreted as the colimit of all the possible finite tensor products of the $k(\alpha)$'s, see e.g. \cite[Prop.~A6.7~b]{eisenbudCommutativeAlgebraView2004}. This implies that, if $a \in A$, then there exists $\alpha_1 , \dots , \alpha_m$ such that $a \in k(\alpha_1) \otimes \dots \otimes k(\alpha_m)$. We want to prove that $A$ is a \emph{real ring}, that is, $-1$ is not a sum of squares in $A$, which we denote by $-1 \notin \Sigma A^2$. Assume that $-1 \in \Sigma A^2$. Then there exist $\alpha_1 , \dots , \alpha_m$ such that, denoting $B \coloneqq k(\alpha_1) \otimes \dots \otimes k(\alpha_m)$, we have $-1 \in \Sigma B^2$. By the amalgamation property for real closed fields, see e.g. \cite{prestelMathematicalLogicModel2011}, there exists a real closed field $R$ and $k$-embeddings $\iota_\alpha \colon k(\alpha) \to R$ for every $\alpha$. Consider the map
\[
    \phi \colon B \longrightarrow R, \quad a_1 \otimes \dots \otimes a_m \longmapsto \iota_{\alpha_1}(a_1) \cdots \iota_{\alpha_m}(a_m)
\]
The field $\quot(B/\ker \phi)$ is a subfield of $R$, and in particular it is a real field. Hence $B$ is a real ideal \cite[Lem.~3.2.16]{Scheiderer2024}: in particular, $-1 \notin \Sigma B^2$, which is a contradiction. Therefore $-1 \notin \Sigma A^2$. This implies that there exists a prime ideal $\p \subset A$ such that $\quot(A/\p)$ is a real field \cite[Lem.~3.2.16]{Scheiderer2024}. Define $\LL$ to be the real closure of $\quot(A/\p)$ with respect to one of its field ordering.
The composition
\[
    k \hookrightarrow k(\alpha) \hookrightarrow A \twoheadrightarrow \quot(A/\p) \hookrightarrow \LL
\]
gives the desired $k$-embeddings.

Now let $\Omega = \Omega(X) \coloneqq \LL((t^\RR))$ be the field of Hahn series with coefficients in $\LL$, which is real closed since $\LL$ is real closed and $\RR$ is divisible \cite[p.~219]{allingFoundationsAnalysisSurreal1987} (see also \cite[Cor.~3.9]{hoevenTransseriesRealDifferential2006}). Notice that there is a canonical embedding $k_\FF((t^{\Gamma_\FF})) \hookrightarrow \Omega$, and a map 
\[
    \FF(\alpha) \longrightarrow \Omega, \quad a \longmapsto \ac_\alpha(a) t^{\val_\alpha(a)}
\]
which makes the diagram 
\begin{center}
    \begin{tikzcd}[sep=large]
     \FF \arrow[r, hook] \arrow[d, "\rv"]
    & \FF(\alpha) \arrow[d, "\rv_\alpha"] \arrow[r] & \Omega \arrow[d, "\rv_\Omega"]\\
     \RV_\FF \arrow[r, hook]& \RV(\alpha) \arrow[r, hook] & \RV_\Omega
    \end{tikzcd}
\end{center}
commute.

\begin{definition}
Using the notations from \Cref{def:val_rv_alpha} and the previous conventions, we define the \emph{fine tropicalization map} as
\[  
    \ftrop \colon \widetilde{X}^\arch \to \RV_\Omega \cong \LL \times \RR^n, \quad  \alpha \longmapsto (\rv_\alpha(x_1(\alpha)) , \dots , \rv_\alpha(x_n(\alpha)))
\]
where we use the embddings $k(\alpha) \hookrightarrow \LL$, and the \emph{tropicalizaiton map}
\[  
    \trop \colon \widetilde{X}^\arch \to \RR^n, \quad  \alpha \longmapsto (\val_\alpha(x_1(\alpha)) , \dots , \val_\alpha(x_n(\alpha)))
\]
Moreover, we equip $\Omega$ with the topology of open sets $U_\Omega$, where $U \subset k_\FF((t^{\Gamma_\FF}))$ is open in the strong topology and $U_\Omega$ is the field extension to $\Omega$. We equip $\RV_\Omega^\times = \Omega^\times/(1+\m_\Omega)$ with the quotient topology, and refer to both of these topologies as the \emph{$\FF$-topology} on $\Omega$ resp. on $\RV_\Omega^\times$.
\end{definition}
\begin{remark}
    The topology defined on $\Omega$ can be strictly coarser than the strong topology induced by the ordering. For example, if $\FF = \RR\{ t \}$ the field of real algebraic Puiseux series, then $\Omega \cong \RR((t^\RR))$. Consider the set $]t^\pi, 2t^\pi[$, which is open in the strong topology of $\Omega$. However, any $\FF$-open set containing an element with valuation $\pi$ must contain the set $\RR^\times \cdot t^\pi$, showing that $]t^\pi, 2t^\pi[$ is not open in the $\FF$-topology.
\end{remark}
\begin{lemma}
    \label{lem:trop_cont}
   Let $\fF$ be a polynomially bounded o-minimal expansion of a real closed field $\FF$, equipped with a nontrivial order compatible valuation $\val \colon \FF \to \Gamma_\FF \cup \{ \infty \}$, and let $X \subset \FF_{>0}^n$ be a definable set. Equip $\Omega$ as before with the $\FF$-topology. The maps $\ftrop \colon \widetilde{X}^\arch \to \RV_\Omega$ and $\trop \colon \widetilde{X}^\arch \to \RR^n$ are continuous.
\end{lemma}
\begin{proof}
    In the following, we identify every $\RV(\alpha)$ with its isomorphic copy in $\RV_\Omega$. Let $U \subset \RV_\Omega$ be $\FF$-open. Then $\rv_\alpha^{-1}(U) = \rv_\alpha^{-1}(U \cap \RV(\alpha))$ is an open subset of $X_{\fF(\alpha)}$ defined over $\FF$. Hence $\widetilde{\rv_\alpha^{-1}(U)}^\arch$ is an open subset of $\widetilde{X}^\arch$. 
    Now notice that
    \[
        \ftrop^{-1}(U) = \bigcup_{\alpha \in \widetilde{X}^\arch} \widetilde{\rv_\alpha^{-1}(U)}^\arch
    \]
    which is an open set in $\widetilde{X}^\arch$, proving that $\ftrop$ is continuous.

    The map $\val_{\rv_\Omega} \colon \RV_{\Omega, > 0} \to \RR$ is continuous: indeed, since $\Gamma_\FF$ is dense in $\RR$, it is sufficient to check that $\val_{\rv_\Omega}^{-1} (]w_1, w_1[)$ is $\FF$-open for $w_1, w_2\in \Gamma_\FF$. But this is clear by the definition of $\FF$-topology and the continuity of $\val_{\rv_\FF}$. Since $\trop = \val_{\rv_\Omega} \circ \ftrop$ the tropicalization map is also continuous.
\end{proof}
   
\begin{proposition}
    \label{prop:trop_proper}
    Let $\fF$ be a polynomially bounded o-minimal expansion of a real closed field $\FF$, equipped with a nontrivial order compatible valuation $\val \colon \FF \to \Gamma_\FF \cup \{ \infty \}$, and let $X \subset \FF_{>0}^n$ be a definable locally closed set. The map $\trop \colon \widetilde{X}^\arch \to \RR^n$ is proper.
\end{proposition}
\begin{proof}
    Let $K \subset \RR^n$ be compact. Then there exists $r \in \RR$ such that \[K \subset \; ]-r, \, r[^n \;\,=\;\,  ] \val(t^{-r}) , \val(t^r)[^n\]
    It follows from the compatibility of the valuation that, if $\alpha \in \widetilde{X}^{\arch}$ is such that $\trop(\alpha) \in K$, we have $t^{r+1} \le t^{r} < x_i(\alpha) < t^{-r} \le t^{-r-1}$ for every $i=1, \dots, n$ (we are using the embedding $\FF \subset \FF(\alpha)$). This implies that \[\trop^{-1}(K) \subset \widetilde{X}^{\arch} \cap \widetilde{[t^{r+1}, t^{-r-1}]^n} = \widetilde{X}^{\max}\cap\widetilde{[t^{r+1}, t^{-r-1}]^n}^{\max} \cong \sper (\cC^{\definable}(X \cap [t^{r+1}, t^{-r-1}]^n))^{\max} \]
    where the equality follows from \Cref{lem:archimedean_functor,lem:closed_bounded}.
    The set $\sper (\cC(X \cap [t^{-r-1}, t^{r+1}]^n))^{\max}$ is a compact (Hausdorff) space, see e.g. \cite[Prop.~3.4.19]{Scheiderer2024}. Since $\trop$ is continuous from \Cref{lem:trop_cont}, $\trop^{-1}(K)$ is a closed subset of a compact space, and hence compact.
\end{proof}
\begin{remark}
    In \cite[Lem.~3.8]{jellRealTropicalizationAnalytification2020} the corresponding properness statement for the tropicalization map is proven in the case of real algebraic varieties, but it not extended to arbitrary semialgebraic sets. On the contrary, in \Cref{prop:trop_proper} we prove properness in the case of arbitrary (locally closed) definable sets: this will simplify the proof of  \Cref{thm:fundamental_theorem}, compared to the proof of \cite[Th.~6.9]{jellRealTropicalizationAnalytification2020}.
\end{remark}

\begin{lemma}
    \label{lem:closure}
    Let $\fF$ be a polynomially bounded o-minimal expansion of a real closed field $\FF$, equipped with a nontrivial order compatible valuation, and let $X \subset \FF_{>0}^n$ be a definable set. Denote $\cl{X} \subset \FF^n_{>0}$ the relative closure of $X$ in $\FF^n_{>0}$. Then we have $$\trop(\widetilde{\cl{X}}^\arch) = \overline{\trop(\widetilde{{X}}^\arch)} = \trop(\widetilde{{X}}^\arch).$$
\end{lemma}
\begin{proof}
    Let $\cl{\widetilde{{X}}^\arch}$ denote the relative closure of $\widetilde{{X}}^\arch$ in $\widetilde{\FF^n_{>0}}^\arch$. We have \[\cl{\widetilde{{X}}^\arch} = \cl{\widetilde{X}} \cap \widetilde{\FF_{>0}}^\arch = \widetilde{\cl{X}} \cap \widetilde{\FF_{>0}}^\arch = \widetilde{\cl{X}}^\arch\]  where:
    \begin{itemize}
        \item the first closure is taken in $\widetilde{\FF_{>0}}^\arch$, the second one in $\widetilde{\FF_{>0}}$, and the third and fourth ones in $\FF_{>0}$;
        \item the second equality follows since the tilde operator commutes with closures, see e.g. \cite[Rem.~2.3]{edmundoSHEAFCOHOMOLOGYOMINIMAL2006}, and the third one from \Cref{lem:archimedean_functor}.
    \end{itemize}
    A proper map is a closed map, and hence the image of the maps $\trop \colon \widetilde{\cl{X}}^\arch \to \RR^n$ and $\trop \colon \widetilde{{X}}^\arch \to \RR^n$ are closed by \Cref{prop:trop_proper}. Therefore:
    \[
        \trop(\widetilde{\cl{X}}^\arch) = \overline{\trop(\widetilde{\cl{X}}^\arch)} = \overline{\trop(\cl{\widetilde{{X}}^\arch})} = \overline{\trop({\widetilde{{X}}^\arch})} = \trop(\widetilde{{X}}^\arch)
    \]
    where the third equality follows because $\trop$ is continuous by \Cref{lem:trop_cont}.
\end{proof}

The next result follows from quantifier elimination and model completeness for real closed fields with a $T$-convex subring \cite{driesTConvexityTameExtensions1995}.

\begin{theorem}[Strong density]
    \label{thm:strong_density}
    Let $\fF$ be a polynomially bounded o-minimal expansion of a real closed field $\FF$, equipped with a nontrivial order compatible valuation $\val \colon \FF \to \Gamma_\FF \cup \{ \infty \}$, and let $X\subset \FF^n_{>0}$ be definable. Then
    \[
        \trop(\widetilde{X}^\arch) \cap \Gamma_\FF^n = \val(X).
    \]
\end{theorem}
\begin{proof}
    The inclusion $\trop(\widetilde{X}^\arch) \cap \Gamma_\FF^n \supset \val(X)$ follows from the embedding $X \hookrightarrow \widetilde{X}^\arch$ and the compatibility of the valuations $\val_\alpha$ with the valuation of $\FF$. For the converse inclusion, let $w \in \trop(\widetilde{X}^\arch) \cap \Gamma_\FF^n $ and $\alpha \in \widetilde{X}^\arch$ such that $\val_\alpha(x_i(\alpha)) = w_i$ for $i = 1, \dots , n$. The extension $\fF \prec \fF(\alpha)$ is elementary, and $\alpha \in X_{\fF(\alpha)}$. \Cref{prop:density} (whose proof was based on \cite{driesTConvexityTameExtensions1995}) implies that there exists $a \in X$ with $\val(a) = w$.
\end{proof}
\subsection{The fundamental theorem}
\label{sec:fundamental_theorems}
In this final section, we complete the proof of \Cref{Thm:A} and \Cref{Thm:B}.
Recall from \eqref{Def:Trop} the definition $\Trop(X) = \overline{\val_\FF(X_\fF)}$.
\begin{theorem}[Fundamental Theorem of Tropical O-minimal Geometry]
    \label{thm:fundamental_theorem}
    Let $\fK \prec \fF \prec \fM$ and $\KK\subset \FF \subset \MM$ be as in Assumption~\ref{ass:elementary_extensions}, and let $X \subset \KK_{>0}^n$ be definable.
    Then the following sets are equal:
    \begin{enumerate}
       \item $\trop ( \widetilde{X_\fF}^\arch)$;
        \item $\Trop(X)$;
          \item  $\overline{\{ \, w \in \Gamma^n_\mathbb{F} \mid \initial_w(X_\fF) \neq \emptyset \, \}}$;
        \item $\val_\MM(X_\fM)$;
        \item $\bigcap_{f \in \mathcal{P}(X)}     \{ \, w \in \RR^n \mid \trop_{r}(f)(1,w) \geq 0 \, \}$.
    \end{enumerate}
\end{theorem}

\begin{proof}
    We first notice that the equality of the sets in (ii), (iii) and in (iv), (v) is a consequence of \Cref{Thm:WeakFundamentalTrop}. We proceed by showing the inclusions (iv) = (v) $\subset$ (ii) = (iii) $\subset$ (i) $\subset$ (iv). 
    
    By Lemma~\ref{Lemma:WeakFundamentalTropOneInclusion}, there exist polynomials $f_1,\dots,f_k \in \mathcal{P}(X)$ such that
    \begin{align*}
 \Trop(X) = \overline{\val(X_\fF)} = \bigcap_{i=1}^k     \{ w \in \RR^n \mid \trop_r(f_i)(1,w) \geq 0\} \supset \bigcap_{f \in \mathcal{P}(X)}    \{ w \in \RR^n \mid \trop_r(f)(1,w) \geq 0\}.
     \end{align*}
    
        We assume in the following that $X$ is locally closed. We show that $\Trop(X)=\overline{\val_\FF(X_\fF)} \subset \trop ( \widetilde{X_\fF}^\arch)$. This is true because
    \[
        \overline{\val(X_\fF)} \subset \overline{\trop(\widetilde{X_\fF}^\arch)} = \trop(\widetilde{X_\fF}^\arch)
    \]
    where the inclusion follows from the embedding $X_\fF \hookrightarrow \widetilde{X_\fF}^\arch$, and the equality from \Cref{lem:closure}.

    We conclude by showing that $\trop ( \widetilde{X_\fF}^\arch) \subset \val_\MM(X_\fM)$. The inclusion of rings $\cC^{\definable}(X_{\fF}) \hookrightarrow \cC^{\definable}(X_{\fM})$ induces a surjective map $\widetilde{X_\fM}^\arch \twoheadrightarrow \widetilde{X_\fF}^\arch$, see \Cref{lem:surjectivity_field_extension}.
    Since this map is surjective and compatible with valuations, we have
    \[
        \trop(\widetilde{X_\fF}^\arch) \subset \trop(\widetilde{X_\fM}^\arch) = \val_{\MM}(X_{\fM})
    \]
    where the equality follows from \Cref{thm:strong_density}.     

    We now remove the assumption that $X$ is locally closed. \Cref{prop:ValImageClosedM} Implies that $\val_\MM(X_\fM) = \val_\MM(\cl{X_\fM}) = \val_\MM(\cl{X}_\fM)$ and $\Trop(X) = \Trop(\cl{X})$. Notice also that $\trop ( \widetilde{X_\fF}^\arch) = \bigcup_\alpha \val_\alpha(X_{\fF(\alpha)})$. Applying again \Cref{prop:ValImageClosedM}, to each $\val_\alpha(X_{\fF(\alpha)})$ we have that $\trop ( \widetilde{X_\fF}^\arch) = \trop ( \widetilde{(\cl{X}_\fF)}^\arch)$. Hence we can always replace $X$ with $\cl{X}$, which is a locally closed definable set, concluding the proof.
\end{proof}

\begin{theorem}[Fundamental Theorem of Fine Tropical O-minimal Geometry]
    \label{thm:fine_fundamental}
    Let $\fM$ be a polynomially bounded o-minimal expansion of a real closed field $\MM$. Assume that $\MM$ has a nontrival order compatible valuation s.t. $\val(\MM^\times) = \RR$ and assume further that the residue field of $\MM$ is $\RR$.
    
    Then for any closed definable subset $X \subset \MM_{>0}^n$, the following sets are equal: 
    \begin{enumerate}
        \item $\ftrop ( \widetilde{X_\fM}^\arch)$;
        \item $\bigcap_{f \in \mathcal{P}(X_\fM)}     \{ z \in \RR_{>0}^n \times \RR^n \mid \trop_{\rv}(f)(z) \geq 0\}$;
        \item $\rv(X_\fM)$;
        \item  $\bigcup_{w \in \RR^n} \initial_w(X_\fM) \times \{ w \}$.
    \end{enumerate}
\end{theorem}

\begin{proof}
    The sets (ii), (iii) and (iv) are equal by \Cref{Thm:WeakFundamentalFineTrop}.
The inclusion (iii) $\subset$ (i) follows directly from the inclusion $X_\fM \subset \widetilde{X_\fM}^\arch$.
To prove that (i) is contained in (ii), let $\alpha \in \widetilde{X_\fM}^\arch$ and consider 
    \begin{align*}
    \rv_\alpha(X_{\fM(\alpha)}) = \bigcap_{f \in \mathcal{P}(X_{\fM(\alpha)})}     \{ z \in \RR_{>0} \times \RR^n \mid \trop_{\rv_\alpha}(f)(z) \geq 0\} \subset \bigcap_{f \in \mathcal{P}(X_{\fM})}     \{ z \in \RR_{>0} \times \RR^n \mid \trop_{\rv}(f)(z) \geq 0\}.
    \end{align*}
    where the inclusion follows from $\mathcal{P}(X_\fM) \subset \mathcal{P}(X_{\fM(\alpha)}) $ and from the fact that $\trop_{\rv_\alpha}(f) = \trop_{\rv}(f)$ for all $f \in \MM[\vb x^\Lambda]$. From~\eqref{Def:Kalpha_Formulas} it follows that $(x_1(\alpha),\dots,x_n(\alpha)) \in X_{\fM(\alpha)}$, which implies that $\ftrop(\alpha)$ lies in (ii).
\end{proof}

\section*{Acknowledgments}
The first author thanks Antongiulio Fornasiero, Noa Lavi, Jana Maříková, and Tamara Servi for their lectures, comments and bibliographical references on model theory and o-minimal geometry, and Erwan Brugallé for his suggestions and explanarions on tropical geometry and on Viro's method.

The first author was partially funded by the Humboldt Research Fellowship for postdoctoral researchers.

This research project began when both authors were hosted at the Max Planck Institute for Mathematics in the Sciences, Leipzig. We would like to thank the institute, and in particular Bernd Sturmfels, for providing us with the opportunity and encouragement to work together. We are also grateful to Marc Burger and Xenia Flamm for inspiring discussions at MPI MiS on the real spectrum compactification of character varieties and ultrafilters, which motivated our exposition of the o-minimal spectrum.

\printbibliography

\end{document}